\documentclass[a4paper,12pt]{article}

\usepackage{amsmath}
\usepackage{amsthm}
\usepackage{enumerate}
\usepackage{amssymb,amsfonts,latexsym,mathtools}
\usepackage{mathrsfs}
\usepackage{bm}
\usepackage{cases}
\usepackage{color}

\newcommand{\Levy}{L\'{e}vy}

\newcommand{\R}{\mathbb{R}}
\newcommand{\F}{\mathscr{F}}

\newcommand{\N}{\mathbb{N}}

\newcommand{\Q}{\mathbb{Q}}

\renewcommand{\P}{\mathbb{P}}
\usepackage{multirow}

\usepackage{url}

\numberwithin{equation}{section}

\makeatletter
\renewcommand\section{\@startsection {section}{1}{\z@}%
{-3.5ex \@plus -1ex \@minus -.2ex}%
{2.3ex \@plus.2ex}%
{\normalfont\large\bf}}
\makeatother

\makeatletter
\renewcommand\subsection{\@startsection {subsection}{1}{\z@}%
{-3.5ex \@plus -1ex \@minus -.2ex}%
{2.3ex \@plus.2ex}%
{\normalfont\normalsize\bf}}
\makeatother

\theoremstyle{plain}
\newtheorem{thm}{Theorem}[section]
\newtheorem{lem}[thm]{Lemma}
\newtheorem{cor}[thm]{Corollary}
\newtheorem{prop}[thm]{Proposition}
\theoremstyle{definition}
\newtheorem{Rem}[thm]{Remark}

\newtheorem{Hyp}[thm]{Hypothesis}
\allowdisplaybreaks

\usepackage{hyperref}
\hypersetup{
setpagesize=false,
 bookmarksnumbered=true,
 bookmarksopen=true,
 colorlinks=true,
 linkcolor=blue,
 citecolor=red,
}

\begin{document}
\begin{center}
\Large \textbf{Fredholm and Sturm--Liouville Type Characterizations of the Limit Measure for the Kac Killing Penalization for \Levy\ Processes}
\end{center}
\begin{center}
Kohki Iba\footnote{
\begin{tabular}[t]{@{}l@{}}
Affiliation: Graduate School of Science, The University of Osaka, Osaka, Japan.\\
E-mail: \url{kohki.iba@gmail.com}
\end{tabular}
}
\end{center}
\begin{abstract}
We study Kac killing penalization for one-dimensional recurrent \Levy\ processes with respect to a finite positive compactly supported measure. We first establish the penalization limit along an independent exponential clock and characterize the function governing the limiting law in terms of the renormalized zero resolvent. We show that this function is uniquely determined by a Fredholm integral equation together with a natural normalization condition. Our main result identifies an equivalent distributional characterization: the limiting function is the unique solution, subject to a prescribed growth condition at infinity, of a Sturm--Liouville-type equation involving the infinitesimal generator of the \Levy\ process and the penalizing measure. The key ingredient is the identity that the renormalized zero resolvent is a fundamental solution of the \Levy\ generator. When the process has a finite second moment, we give an equivalent increment condition at infinity. These results extend the classical Brownian Sturm--Liouville characterization and, in particular, provide its nonlocal analogue for recurrent stable processes. Finally, we identify the limiting process as a Doob h-transform of the associated Feynman--Kac subprocess and obtain a path decomposition.
\end{abstract}


\section{Introduction}
\subsection{Penalization Problem}
The \emph{penalization problem} concerns the following long-time limit:
\begin{align}
\label{penal-lim}
\lim_{t\to \infty}\frac{\P_x[F_s\cdot \Gamma_t]}{\P_x[\Gamma_t]}\qquad \text{for}\ F_s\in b\F_s,
\end{align}
where $((X_t)_{t\ge 0},(\F_t)_{t\ge 0},(\P_x)_{x\in \R})$ is a Markov process, $\P_x[\cdot]$ denotes integration with respect to the measure $\P_x$, and $(\Gamma_t)_{t\ge 0}$ is a non-negative process referred to as a \emph{weight process}. By choosing $\Gamma$ to be a suitable functional of the underlying stochastic process---for example, a multiplicative functional or a functional constructed from an additive functional---one can give a probabilistic interpretation to the process arising in the limit.

For example, consider $\Gamma_t=f(L_t^0)$, where $(L_t^0)_{t\ge 0}$ denotes the local time at $0$ and $f$ is a suitable function. If $f$ is a non-negative decreasing function, such as $f(l)=e^{-l}$, then $L^0$ increases whenever the original process visits $0$, and consequently $\Gamma$ decreases as local time accumulates. One therefore expects the limit in the penalization problem (\ref{penal-lim}) to describe a process that tends to accumulate relatively little local time at $0$. This type of penalization is referred to as the \emph{local time penalization}. The local time penalization is based on the local time at the single point $0$. The extension of this idea from the local time at a single point to local times on a set, implemented through a suitable weight process, is referred to as the \emph{Kac killing penalization} or the \emph{Feynman--Kac penalization}. More precisely, for an appropriate positive measure $\mu$, one considers the weight
\begin{align}
\label{weight}
\Gamma_t^\mu:=\exp \left(-\int_\R L_t^y\,\mu(dy) \right).
\end{align}

Roynette--Vallois--Yor \cite{Roynette-Vallois-Yor-1,Roynette-Vallois-Yor-2} studied the local time penalization and the Kac killing penalization for one-dimensional Brownian motion:

\begin{thm}[Theorem 5.1 of Roynette--Vallois--Yor \cite{Roynette-Vallois-Yor-1}]
\label{Brown-penal-thm}
Let $X$ be a standard Brownian motion. Suppose that $\mu$ is a finite positive Radon measure satisfying $\int_\R (1+|y|)\,\mu(dy)<\infty$. Then there exists a function $\phi^\mu$ such that
\begin{align}
\phi^\mu(x)=\lim_{t\to \infty}\sqrt{t}\P_x\Big[\Gamma_t^\mu\Big]\qquad \text{for}\ x\in \R
\end{align}
and
\begin{align}
\lim_{t\to \infty}\frac{\P_x[F_s\cdot \Gamma_t^\mu]}{\P_x[\Gamma_t^\mu]}=\P_x\left[F_s\cdot \frac{\phi^\mu(X_s)\Gamma_s^\mu}{\phi^\mu(x)}\right]\qquad \text{for}\ F_s\in b\F_s,
\end{align}
Moreover, the function $\phi^\mu$ is the unique solution to the Sturm--Liouville equation
\begin{align}
\label{Brown-SL}
\frac{\Delta}{2}f=f\cdot \mu\qquad \text{in the weak sense,}
\end{align}
with boundary conditions
\begin{align}
\label{Brown-growth}
  \lim_{x\to \infty}f'(x)=-\lim_{x\to -\infty}f'(x)=\sqrt{\frac{2}{\pi}}.
\end{align}
\end{thm}

Najnudel \cite{Najnudel} further generalized the Kac killing penalization for Brownian motion by considering weight processes of the form $\Gamma_t=F((L_t^y)_{y\in \R})$.

Yano--Yano--Yor \cite{Yano-Yano-Yor} studied the local time penalization and the Kac killing penalization for one-dimensional symmetric stable processes:

\begin{thm}[Corollary 8.2 of Yano--Yano--Yor \cite{Yano-Yano-Yor}]
\label{Stable-penal-thm}
Let $X$ be a symmetric $\alpha$-stable process with $1<\alpha\le 2$. Suppose that $\mu$ is a finite positive Radon measure satisfying $\int_\R (1+|y|^{\alpha-1})\,\mu(dy)<\infty$. Then
\begin{align}
\lim_{t\to \infty}\frac{\P_x[F_s\cdot \Gamma_t^\mu]}{\P_x[\Gamma_t^\mu]}=\mathscr{P}_x\left[F_s\cdot \frac{\Gamma_\infty^\mu}{\mathscr{P}_x[\Gamma_\infty^\mu]}\right]\qquad \text{for}\ F_s\in b\F_s,
\end{align}
where $\mathscr{P}_x$ is the universal $\sigma$-finite measure (see Section \ref{S5} for details).
\end{thm}

Although this result establishes the penalization limit itself, it does not provide an equation-based characterization of the function appearing in the limit analogous to the Brownian characterization in (\ref{Brown-SL}). Indeed, Remark 8.4 of their paper \cite{Yano-Yano-Yor} states that such a characterization was not known even for symmetric stable processes. In the present paper, we address this problem for a more general class of recurrent \Levy\ processes under the assumption that the penalizing measure $\mu$ is compactly supported. More precisely, we characterize the function appearing in the limit through a Sturm--Liouville equation in the sense of distributions, formulated in terms of the infinitesimal generator of the \Levy\ process. This resolves, in the compact-support setting, the question left open in the preceding work.

For general \Levy\ processes, Takeda--Yano \cite{Takeda-Yano} studied the local time penalization problem. Concerning the Kac killing penalization, Iba--Yano \cite{Iba-Yano} treated the case $\mu=a_1\delta_{a_1}+a_2\delta_{a_2}$, while Iba \cite{Iba-penalization} studied the case $\mu=a_1\delta_{a_1}+\cdots +a_n\delta_{a_n}$. The present work extends these results by investigating the Kac killing penalization problem associated with a finite positive measure $\mu$ having compact support. For finitely supported penalizing measures, the problem reduces to a finite-dimensional linear system. For a general compactly supported measure, this reduction is no longer available: one is led instead to a compact integral operator, whose eigenvalue $1$ must be treated separately. A key ingredient of our argument is an energy inequality for the renormalized zero resolvent, which allows us to combine the Fredholm alternative with probabilistic information.

Although the preceding discussion has focused on penalization problems for \Levy\ processes, such problems have also been extensively studied for many other classes of stochastic processes. For example, Debs \cite{Debs} studied random walks, Salminen--Vallois \cite{Salminen-Vallois} and Profeta--Yano--Yano \cite{Profeta-Yano-Yano} studied diffusion processes, and Abraham--Debs \cite{Abraham-Debs} studied Galton-Watson processes. Moreover, in the setting of local time penalization, the choice $f(l)=1_{\{l=0\}}$ gives rise to the problem of \emph{conditioning to avoid zero}. This conditioning problem has a long history and an extensive literature, which we do not attempt to review here.

Finally, we emphasize that the exponent in (\ref{weight}) has a negative sign. Since $\mu$ is a positive measure, the resulting weight process is non-increasing, and trajectories along which the original process visits the relevant set are increasingly suppressed. The penalized process is therefore expected to be less likely to visit that set. One may instead consider a weight with a positive exponent, that is,
\begin{align}
\Gamma_t=\exp \left(\int_\R L_t^y\,\mu(dy)\right).
\end{align}
In this case, the weight process is non-decreasing and favors trajectories along which the original process visits the set. In this sense, such a procedure is more appropriately regarded as a reward rather than a penalty, and in the present paper we reserve the term penalization for the negative-exponent setting. We refer to this positive-exponent setting as the \emph{Kac creation}. This problem has also been studied extensively. Takeda \cite{Takeda} and Wada \cite{Wada} considered symmetric stable processes, while Abildaev \cite{Abildaev} studied symmetric \Levy\ processes in the case $\Gamma_t=\exp (\lambda L_t^a)$ for $\lambda>0$ and $a\in \R$.

\subsection{\Levy\ Process, Resolvent, and Generator}
We denote by $((X_t)_{t\ge 0},(\P_x)_{x\in \R})$ a one-dimensional \Levy\ process such that $\P_x(X_0=x)=1$ for each $x\in \R.$ For $\lambda\in \R$, we denote the characteristic exponent of $X$ by $\Psi$, that is, $\Psi(\lambda)$ satisfies
\begin{align}
\P_0[e^{i\lambda X_t}]=e^{-t\Psi(\lambda)}\qquad \text{for}\ t\ge 0.
\end{align}
By the \Levy--Khintchine formula (see, e.g., Theorem 8.1 of \cite{Sato}), there exist constants $a\in \R$ and $\sigma\ge 0$ and a measure $\Pi$ on $\R$ satisfying $\Pi(\{0\})=0$ and $\int_\R(1\wedge y^2)\,\Pi(dy)<\infty$, such that
\begin{align}
  \Psi(\lambda)=ia\lambda +\frac{1}{2}\sigma^2\lambda^2+\int_\R \Big(1-e^{i\lambda y}+i\lambda y 1_{\{|y|<1\}}\Big)\,\Pi(dy).
\end{align}
The triplet $(a,\sigma^2,\Pi)$ is called the \Levy\ triplet. The process $\widehat{X}:=-X$ is called the dual process of $X$. It is clear that the \Levy\ triplet of $-X$ is given by $(-a,\sigma^2,\widehat{\Pi})$, where $\widehat{\Pi}$ is a measure with $\widehat{\Pi}(dy):=\Pi(-dy)$.

For $f\in L^\infty(\R)$, we define
\begin{align}
  P_tf(x):=\P_x[f(X_t)]\qquad \text{for}\ x\in \R\ \text{and}\ t\ge 0.
\end{align}
It is known that $(P_t)_{t\ge 0}$ restricted on $C_0(\R)$ is a strongly continuous semigroup on $C_0(\R)$, where $C_0(\R)$ denotes the space of continuous functions vanishing at infinity, equipped with the max norm $\|\cdot\|_{\max}$ (see, e.g., Theorem 31.5 of \cite{Sato}). For $q>0$, we define the $q$-resolvent operator $U^q$ by
\begin{align}
U^q f(x):=\P_x\left[\int_0^\infty e^{-qt}f(X_t)\,dt\right]\qquad \text{for}\ f\in L^\infty(\R)\ \text{and}\ x\in \R.
\end{align}
Let $\mathscr{L}$ denote the infinitesimal generator of $(P_t)$ to $C_0(\R)$ and $\mathscr{D}(\mathscr{L})\subset C_0(\R)$ denote its domain, that is, 
\begin{align}
  \mathscr{L}f:=\lim_{t\to 0+}\frac{1}{t}(P_tf-f)\ \text{in\ norm}\qquad \text{for}\ f\in \mathscr{D}(\mathscr{L}),
\end{align}
where $\mathscr{D}(\mathscr{L})$ denotes the class of functions $f$ which admits the norm limit. For $q>0$, the operator $qI-\mathscr{L}$ is known as the inverse of the $q$-resolvent operator $U^q$:
\begin{align}
\label{qI-L}
U^q(qI-\mathscr{L})=I\qquad \text{on}\ \mathscr{D}(\mathscr{L}),
\end{align}
where $I$ denotes the identity operator (see, e.g., p. 23 of \cite{Bertoin}). Let $C_c^\infty(\R)$ denote the space of infinitely differentiable functions with compact support. It is known that $C_c^\infty(\R)\subset \mathscr{D}(\mathscr{L})$ and that for $f\in C_c^\infty(\R)$,
\begin{align}
\label{Lf-equation}
  \mathscr{L}f(x)=-af'(x)+\frac{\sigma^2}{2}f''(x)+\int_\R \Big(f(x+y)-f(x)-yf'(x)1_{\{|y|< 1\}}\Big)\,\Pi(dy)
\end{align}
(see, e.g., Theorem 31.5 of \cite{Sato}). It is well known that the infinitesimal generator $\mathscr{L}$ is given by $\Delta/2$ in the case of Brownian motion, whereas for a symmetric $\alpha$-stable process it is given by $-(-\Delta)^{\frac{\alpha}{2}}$, the fractional Laplacian.

Throughout this paper, we impose the following hypothesis.

\begin{Hyp}
\label{Hyp}
The process $(X,\P_0)$ is recurrent and the following condition holds:
\begin{align}
\int_0^\infty \left|\frac{1}{q+\Psi(\lambda)}\right|\,d\lambda<\infty\qquad \text{for}\ q>0 \tag{\textbf{A}}
\end{align}
\end{Hyp}

Note that under this condition, we have
\begin{align}
\label{Psi-zero}
\{\lambda:\ \Psi(\lambda)=0\}=\{0\}.
\end{align}
The proof is deferred to Appendix \ref{Appendix}. Under Hypothesis \ref{Hyp}, for $q>0$, it is known that $X$ has a bounded continuous resolvent density $r_q$ which satisfies
\begin{align}
  U^qf(x)=\int_\R f(y)r_q(y-x)\,dy\qquad \text{for}\ f\in L^\infty(\R)\ \text{and}\ x\in \R
\end{align}
(see, e.g., Theorems II.16 and II.19 of \cite{Bertoin}). The resolvent density can be expressed as
\begin{align}
\label{rq-integral}
r_q(x)=\frac{1}{\pi}\int_0^\infty \Re \left(\frac{e^{-i\lambda x}}{q+\Psi(\lambda)}\right)\,d\lambda \qquad \text{for}\ x\in \R\ \text{and}\ q>0,
\end{align}
(see, e.g., Lemma 2 \cite{Winkel} and Corollary 15.1 of \cite{Tsukada}). Moreover, the following asymptotic behaviors are known:
\begin{align}
\label{rq-lim}
  \lim_{q\to 0+}r_q(0)=\infty,\qquad \lim_{q\to 0+}qr_q(0)=0,
\end{align}
(see, e.g., Theorem 37.5 of \cite{Sato} and Lemma 15.5 of \cite{Tsukada}).

Under our hypothesis, a local time at $y\in \R$ exists. We denote it by $(L_t^y)_{t\ge 0}$. We may and do assume that $L^y$ is normalized to satisfy
\begin{align}
\P_x\left[\int_0^\infty e^{-qt}\,dL_t^y\right]=r_q(y-x)\qquad \text{for}\ x\in \R\ \text{and}\ q>0,
\end{align}
(see, e.g., Section V of \cite{Bertoin}). It is known that $L^y$ is a positive continuous additive functional. Moreover, for every finite positive measure $\mu$, the process $(\int_\R L_t^y\,\mu(dy))_{t\ge 0}$ is a positive continuous additive functional (see, e.g., Theorem V.5 of \cite{Bertoin}). Thus, for such $\mu$, the process (\ref{weight}) is a continuous multiplicative functional.

\subsection{Renormalized Zero Resolvent}
We define
\begin{align}
\label{hq-def}
  h_q(x):=r_q(0)-r_q(-x)=\frac{1}{\pi}\int_0^\infty \Re \left(\frac{1-e^{i\lambda x}}{q+\Psi(\lambda)}\right)\,d\lambda\qquad \text{for}\ x\in \R,
\end{align}
where the second equality follows from (\ref{rq-integral}). Note that this function $h_q$ is non-negative. For any $x\in \R$, the limit
\begin{align}
  h(x):=\lim_{q\to 0+}h_q(x)
\end{align}
exists and is finite. Moreover, this convergence is uniform on compact sets (see Theorem 1.1 of Takeda--Yano \cite{Takeda-Yano}). We call this limit function $h$ the \emph{renormalized zero resolvent}. It is known that $h$ is non-negative, continuous, and subadditive, that is,
\begin{align}
  h(x+y)\le h(x)+h(y)\qquad \text{for}\ x,y\in \R,
\end{align}
(see Theorem 1.1 of \cite{Takeda-Yano}). Moreover, the following asymptotic behaviors for $h$ are known:
\begin{align}
  \lim_{x\to \pm \infty}&\frac{h(x)}{|x|}=\frac{1}{\P_0[X_1^2]},\\
  \label{h-limit2}
  \lim_{x\to \pm \infty}&\Big(h(x+y)-h(x)\Big)=\pm \frac{y}{\P_0[X_1^2]}\qquad  \text{for}\ y\in \R,
\end{align}
where we use the convention that $1/\infty=0$ (see Theorem 1.2 of \cite{Takeda-Yano}).

The explicit form of the renormalized zero resolvent is known for several specific \Levy\ processes. For example, if $X$ is standard Brownian motion, then
\begin{align}
  h(x)=|x|,
\end{align}
(see, e.g., Example 5.1 of \cite{Takeda-Yano}). If $X$ is a recurrent $\alpha$-stable process, then
\begin{align}
  h(x)=\frac{1}{K(\alpha,\beta)}(1-\beta\,\mathrm{sgn}(x))|x|^{\alpha-1},
\end{align}
where $K(\alpha,\beta)>0$ and $\beta\in [-1,1]$ are constants (see Section 5 of Yano \cite{Yano}). If $X$ is a spectrally negative \Levy\ process, then
\begin{align}
  h(x)=W(x)-\frac{x}{\P_0[X_1^2]},
\end{align}
where $W$ denotes the zero-scale function of $X$ (see Example 5.2 of Pant\'{i} \cite{Panti}).

It is known that the renormalized zero resolvent $h$ is a harmonic function for the process killed upon hitting $0$, that is,
\begin{align}
h(x)=\P_x\Big[h(X_t),\ t<T_0\Big]\qquad \text{for}\ x\in \R\ \text{and}\ t\ge 0,
\end{align}
where $T_0$ denotes the first hitting time of $0$ (see Theorem 2.2 of \cite{Panti}, or see also Theorem 1.1 of Yano \cite{Yano-Exc}). This result can be extended to the case of a compact set $K$. We define the function
\begin{align}
  \varphi_K(x):=\lim_{q\to 0+}r_q(0)\P_x(\bm{e}_q<T_K),
\end{align}
where $T_K$ denotes the first hitting time of $K$ and $\bm{e}_q$ denotes an exponential random variable with parameter $q>0$ that is independent of the process $X$. It is known that $\varphi_K$ is a harmonic function for the process killed upon hitting $K$ (see Theorem 1.3 of Iba \cite{Iba-conditioning}). Moreover, this function $\varphi_K$ admits the following integral representation involving the renormalized zero resolvent $h$; there exists a unique probability measure $\mu_K$ supported on $K$, together with a unique non-negative constant $k(K)$ such that
\begin{align}
\label{Iba-Rivero}
\varphi_K(x)=\int_K h(x-y)\mu_K(dy)-k(K)\qquad \text{for}\ x\in \R,
\end{align}
(see Theorem 1.2 of Iba--Rivero \cite{Iba-Rivero}). The measure $\mu_K$ is called the \emph{equilibrium measure} and $k(K)$ is called the \emph{Robin constant}.

\subsection{Main Results}
We assume that $\mu$ is a nonzero finite positive measure with compact support. Let $K:=\mathrm{supp}(\mu)$. Then we consider the following weight:
\begin{align}
\Gamma_t^\mu :=\exp \left(-\int_K L_t^y\,\mu(dy)\right)\qquad \text{for}\ t\ge 0.
\end{align}

Let $C(K)$ denote the space of continuous functions on $K$, equipped with the max norm. We define an operator $T^\mu:C(K)\to C(K)$ as follows:
\begin{align}
T^\mu f(x):=\int_K f(y)h(x-y)\,\mu(dy)\qquad \text{for}\ f\in C(K).
\end{align}

In this setting, our main result on the penalization problem associated with the weight $\Gamma^\mu$ is as follows:

\begin{thm}
\label{mainthm1}
There exists a function $\varphi^\mu$ such that
\begin{align}
\label{tsuika-lim}
\varphi^\mu(x)=\lim_{q\to 0+}r_q(0)\P_x\Big[\Gamma_{\bm{e}_q}^\mu\Big]\qquad \text{for}\ x\in \R
\end{align}
and
 \begin{align}
 \label{main-penallim}
   \lim_{q\to 0+}\frac{\P_x[F_t\cdot \Gamma_{\bm{e}_q}^\mu]}{\P_x[\Gamma_{\bm{e}_q}^\mu]}=\P_x\left[F_t\cdot\frac{\varphi^\mu(X_t)\Gamma_t^\mu}{\varphi^\mu(x)} \right]\qquad \text{for}\ t\ge0,\ F_t\in b\F_t.
 \end{align}
Moreover, the function $\varphi^\mu$ satisfies
 \begin{align}
   \varphi^\mu(x)=C^\mu +\int_K \varphi^\mu(y)h(x-y)\,\mu(dy)>0\qquad \text{for}\ x\in \R
 \end{align}
 and the pair $(C^\mu,\varphi^\mu|_K)$ is the unique pair $(C,f)\in \R\times C(K)$ satisfying the Fredholm equation
 \begin{align}
 \label{Levy-Fredholm}
   (I-T^\mu)f=C\qquad \text{for}\ f\in C(K),
 \end{align}
 with the normalization condition
 \begin{align}
 \mu[f]:=\int_K f(y)\,\mu(dy)=1.
 \end{align}
In particular, if $\mu=c\mu_K$ for $c>0$, then we have
\begin{align}
\varphi^{c\mu_K}(x)=\varphi_K(x)+\frac{1}{c}\qquad \text{for}\ x\in \R.
\end{align}
\end{thm}

The proof of this theorem is given in Section \ref{S3}.

\begin{Rem}
For Brownian motion (Theorem \ref{Brown-penal-thm}) and stable processes (Theorem \ref{Stable-penal-thm}), the penalization limits are formulated by letting the deterministic time $t\to\infty$. In contrast, our main result concerns a limit along the random times $\bm{e}_q$ as $q\to 0+$, and the two formulations may therefore appear different at first sight. Appendix \ref{Appendix2} shows more generally that, whenever $q\mapsto r_q(0)$ is regularly varying at $0+$ with index $-\rho$ for some $\rho\in(0,1)$, the same penalized process is obtained along deterministic times. This condition includes Brownian motion and recurrent symmetric stable processes.
\end{Rem}

The preceding theorem shows that the function arising from Kac killing penalization is characterized by a Fredholm equation. This characterization, however, appears substantially different from equation (\ref{Brown-SL}) in the Brownian setting. We therefore proceed to reformulate (\ref{Levy-Fredholm}).

Let $(C_c^\infty(\R))'$ denote the continuous dual space of $C_c^\infty(\R)$ (see, e.g., Chapter 3 of \cite{Duistermaat-Kolk}). For $g\in {C_c^\infty(\R)}$ and $T\in (C_c^\infty(\R))'$, we denote by $\langle T,g\rangle$ the duality pairing between $(C_c^\infty(\R))'$ and $C_c^\infty(\R)$. It is known that $L^1(\R)$ and the space of Radon measures $\mathscr{M}(\R)$ can be naturally embedded into $(C_c^\infty(\R))'$. More precisely, for $f\in L^1(\R)$ and $\nu \in \mathscr{M}(\R)$, 
\begin{align}
\langle f,g\rangle :=\int_\R f(x)g(x)\,dx,\qquad \langle \nu,g\rangle :=\int_\R g(x)\,\nu(dx)
\end{align}
(see, e.g., Theorem 3.5 and Theorem 20.1 of \cite{Duistermaat-Kolk}). We will see in Lemma \ref{lem4.1} that $\mathscr{L}f$ belongs to $L^1(\R)$ for $f\in C_c^\infty(\R)$. Then, it is easy to see that
\begin{align}
\label{def-dual}
\langle \mathscr{L}f,g\rangle =\int_\R f(x)\widehat{\mathscr{L}}g(x)\,dx\qquad \text{for}\ f,g\in C_c^\infty(\R),
\end{align}
where $\widehat{\mathscr{L}}$ denotes the infinitesimal generator of the dual \Levy\ process $\widehat{X}$, that is, we have
\begin{align}
\widehat{\mathscr{L}}g(x)=ag'(x)+\frac{\sigma^2}{2}g''(x)+\int_\R \Big(g(x+y)-g(x)-yg'(x)1_{\{|y|< 1\}}\Big)\,\widehat{\Pi}(dy).
\end{align}
Motivated by (\ref{def-dual}), we extend the action of the operator $\mathscr{L}$ to the class
\begin{align}
\mathscr{A}:=\Big\{f:\ f\ \text{is\ measurable and}\ f\cdot \widehat{\mathscr{L}}g\in L^1(\R)\ \text{for}\ g\in C_c^\infty(\R)\Big\}
\end{align}
so that, for $f\in \mathscr{A}$, we define $\mathscr{L} f\in (C_c^\infty(\R))'$ by
\begin{align}
\label{Def-exL}
\langle \mathscr{L}f,g\rangle:=\int_\R f(x)\widehat{\mathscr{L}}g(x)\,dx\qquad \text{for}\ g\in C_c^\infty(\R).
\end{align}
Note that Lemma \ref{lem4.1} implies $1\in \mathscr{A}$ and Lemma \ref{lem4.3} implies $h\in \mathscr{A}$. Thus, we have
\begin{align}
  \mathscr{A}_0:=\left\{f:\ f\ \text{is measurable and }\frac{|f(x)|}{1+h(x)}\ \text{is bounded on $\R$}\right\}\subset \mathscr{A}.
\end{align}
The extension of $\mathscr{L}$ was discussed by Berger--Schilling \cite{Berger-Schilling} and Grzywny--Kwa\'{s}nicki \cite{Grzywny-Kwasnicki}. They studied the Liouville theorem for $\mathscr{L}$.

With these conventions, we obtain the following equivalence between the two characterizations of $\varphi^\mu$.

\begin{thm}
\label{mainthm2}
Let $f\in C(\R)$. The following assertions are equivalent:
\begin{enumerate}
\item There exists a constant $C$ such that the function $f$ satisfies the Fredholm equation
\begin{align}
\label{main-Fredeq}
f(x)
=C+\int_K f(y)h(x-y)\,\mu(dy)\qquad \text{for}\ x\in \R,
\end{align}
with the normalization condition
\begin{align}
\label{main-Frednorm}
  \mu[f]=1.
\end{align}
\item The function $f$ satisfies the growth condition
\begin{align}
\label{main-SLgrowth}
f-h\in C_b(\R)
\end{align}
(which implies $f\in \mathscr{A}_0$) and the Sturm--Liouville equation
\begin{align}
\label{main-SLeq}
\mathscr{L}f
=f\cdot \mu,
\end{align}
that is,
\begin{align}
\langle \mathscr{L} f,g\rangle =\int_\R g(x) f(x)\,\mu(dx)\qquad \text{for}\ g\in C_c^\infty(\R).
\end{align}
\end{enumerate}
If either, and hence both, of these conditions holds, then we have $f=\varphi^\mu.$
\end{thm}

The proof of this theorem will be given in Section \ref{S4}.

We may rewrite the Sturm--Liouville equation (\ref{main-SLeq}) as
\begin{align}
  (-\mathscr{L}+\mu)f=0.
\end{align}
The operator $\mathscr{H}:=-\mathscr{L}+\mu$ may be called a \emph{Schr\"{o}dinger-type operator}. Thus, a solution $f$ of (\ref{main-SLeq}) may be regarded as a harmonic function associated with the Schr\"{o}dinger-type operator $\mathscr{H}$. In the case of Brownian motion, this theory is classical (see, e.g., \cite{Chung-Zhao}). Takeda \cite{Takeda-harmonic} studied the corresponding problem for recurrent symmetric stable processes.

If the process has a finite second moment, the growth condition associated with the Sturm--Liouville equation admits the following alternative characterization:

\begin{thm}
\label{mainthm3}
Assume that $\P_0[X_1^2]<\infty$. If a function $f\in C(\R)$ satisfies the Sturm--Liouville equation (\ref{main-SLeq}), then the following growth conditions are equivalent:
\begin{enumerate}
  \item[i.] (\ref{main-SLgrowth}) holds.
  \item[ii.] For every $z\in \R$,
  \begin{align}
  \label{Levy-growth}
    \lim_{x\to \pm \infty}\Big(f(x+z)-f(x)\Big)=\pm \frac{z}{\P_0[X_1^2]}.
  \end{align}
\end{enumerate}
\end{thm}

The proof of this theorem will be given in Section \ref{S4}.

\begin{Rem}
Suppose that $X$ is a standard Brownian motion. Then we have
\begin{align}
\mathscr{L}=\frac{\Delta}{2},\qquad h(x)=|x|,\qquad \mathbb{P}_0[X_1^2]=1.
\end{align}
Hence, the Sturm--Liouville equation in the characterization of Roynette--Vallois--Yor coincides with the equation (\ref{main-SLeq}). The corresponding growth conditions (\ref{Brown-growth}) and (\ref{Levy-growth}) are equivalent but differ by a multiplicative constant. This discrepancy reflects the different normalizations naturally associated with the two procedures: Theorem \ref{Brown-penal-thm} considers the deterministic-time limit $t\to\infty$, whereas Theorem \ref{mainthm1} considers the limit along an independent exponential clock $\bm{e}_q$ as $q\to 0+$. More precisely, if $\phi^\mu$ denotes the function in Theorem \ref{Brown-penal-thm} and $\varphi^\mu$ denotes that in Theorem \ref{mainthm1}, then
\begin{align}
\varphi^\mu (x)=\sqrt{\frac{\pi}{2}}\phi^\mu (x).
\end{align}
Nevertheless, the penalized law in (\ref{penal-lim}) depends on these functions only through the ratio
\begin{align}
  \frac{\phi^\mu(X_t)}{\phi^\mu(x)}=\frac{\varphi^\mu(X_t)}{\varphi^\mu(x)}.
\end{align}
The multiplicative constant therefore cancels, and the two penalization procedures yield the same penalized law.
\end{Rem}

Finally, we consider the process obtained as the penalization limit (\ref{main-penallim}).

We define $\F_\infty:=\sigma(\bigcup_{t\ge 0}\F_t)$. For $x\in \R$, let $\Q_x^\mu$ denote the \emph{penalized law} arising in (\ref{main-penallim}), that is, the measure $\Q_x^\mu$ on $\F_\infty$ satisfies
\begin{align}
\frac{d\Q_x^\mu}{d\P_x}\Big|_{\F_t}=\frac{\varphi^\mu(X_t)\cdot\Gamma_t^\mu}{\varphi^\mu(x)}\qquad \text{for}\ t\ge 0.
\end{align}
For the existence of such a measure, see, e.g., Theorem A.1 of \cite{Yano-penal}. Its transition semigroup $Q_t^\mu$ is given by
\begin{align}
Q_t^\mu f(x):=\frac{1}{\varphi^\mu(x)}\P_x\Big[f(X_t)\cdot\varphi^\mu(X_t)\Gamma_t^\mu\Big]\qquad \text{for}\ f\in L^\infty(\R).
\end{align}
This is the Doob h-transform of the Feynman--Kac subprocess with respect to $\varphi^\mu$. Thus, under $\Q_x^\mu$, the process is a strong Markov process (see, e.g., Proposition 3.12 of \cite{B-G} and Theorem 11.9 of \cite{C-W}).

When the original \Levy\ process admits transition densities, the penalized process can be represented as follows:

\begin{thm}
\label{main-thm4}
Assume that $(X,\mathbb{P}_x)$ admits transition densities $p_u(\cdot)$ for $u>0$. For $x\in \R$, we then have
\begin{align}
  \Q_x^\mu=\frac{1}{\varphi^\mu(x)}\int_0^\infty \Big((\Gamma_u^\mu\cdot\P_{x,0}^u)\bullet (\Gamma_\infty^\mu\cdot\P_0^{(0)})\Big)p_u(-x)du+\frac{h(x)\cdot \Gamma_\infty^\mu}{\varphi^\mu(x)}\cdot\P_x^{(0)},
\end{align}
where $\mathbb{P}_{x,0}^u$ denotes the \Levy\ bridge from $X_0=x$ to $X_u=0$, $\mathbb{P}_x^{(0)}$ denotes the law of the \Levy\ process conditioned to avoid zero, and the symbol $\bullet$ denotes concatenation.
\end{thm}

The proof of this theorem is given in Section \ref{S5}.

This theorem provides an intuitive description of the sample-path behavior under $\mathbb{Q}_x^\mu$. Roughly speaking, the process evolves as a \Levy\ process while being weighted so as to avoid the regions to which $\mu$ assigns mass. After its last visit to zero, it moves away from zero forever as a \Levy\ process conditioned to avoid zero; alternatively, it may behave as such a conditioned \Levy\ process from the outset.

\subsection{Organization}
The remainder of this paper is organized as follows. In Section \ref{S2}, we prove that the Fredholm equation (\ref{Levy-Fredholm}) on the compact set $K$ admits a unique solution. In Section \ref{S3}, we establish the penalization limit (\ref{main-penallim}), thereby completing the proof of Theorem \ref{mainthm1}. In Section \ref{S4}, we prove Theorems \ref{mainthm2} and \ref{mainthm3}. More precisely, in Subsection \ref{S4.1}, we derive the Sturm--Liouville equation from the Fredholm equation, while in Subsection \ref{S4.2}, we prove the converse implication. In Subsection \ref{S4.3}, we prove Theorem \ref{mainthm3}. In Section \ref{S5}, we present several properties of the penalized process arising from (\ref{main-penallim}) and prove Theorem \ref{main-thm4}. Finally, Appendix \ref{Appendix} contains the proofs of the lemmas omitted from the main text, while Appendix \ref{Appendix2} establishes the constant clock penalization limit.


\section{Fredholm Equation}
\label{S2}
For $q>0$, we first consider the function
\begin{align}
\varphi_q^\mu(x):= r_q(0)\P_x\Big[\Gamma_{\bm{e}_q}^\mu\Big]\qquad \text{for}\ x\in \R.
\end{align}
We define an operator $T_q^\mu :C(K)\to C(K)$ by
\begin{align}
T_q^\mu f(x)&:=\int_K f(y)h_q(x-y)\,\mu(dy)\qquad \text{for}\ f\in C(K).
\end{align}

In this case, the function $\varphi_q^\mu$ also satisfies a Fredholm equation similar to (\ref{Levy-Fredholm}).

\begin{lem}
The function $\varphi_q^\mu(x)$ is continuous on $\R$. Moreover, there exists a constant $C_q^\mu$ such that $\varphi_q^\mu|_K$ satisfies the following equation:
\begin{align}
\label{phiq-equation}
(I-T_q^\mu)(\varphi_q^\mu|_K)=C_q^\mu,\qquad \mu[\varphi_q^\mu|_K]+\frac{C_q^\mu}{r_q(0)}=1.
\end{align}
\end{lem}
\begin{proof}
For $q>0$ and $f\in L^\infty(\R)$, we define
\begin{align}
  V_\mu^q f(x)&:=\P_x\left[\int_0^\infty e^{-qt}f(X_t)\Gamma_t^\mu \,dt\right],\\
  U^q(f\mu)(x)&:=\int_\R r_q(y-x)f(y)\,\mu(dy).
\end{align}
By the Feynman--Kac formula (see, e.g., p. 139 of \cite{Bertoin}), we have
\begin{align}
\label{FC-formula}
  U^q\Big((V_\mu^q f)\mu\Big)(x)=U^qf(x)-V_\mu^q f(x).
\end{align}
Since 
\begin{align}
  V_\mu^q 1(x)=\int_0^\infty e^{-qt}\P_x\Big[\Gamma_t^\mu\Big]\,dt=\frac{1}{q}\P_x\Big[\Gamma_{\bm{e}_q}^\mu\Big],
\end{align}
by setting $f\equiv 1$ in (\ref{FC-formula}), we have
\begin{align}
\label{lem1-eq1}
  \P_x\Big[\Gamma_{\bm{e}_q}^\mu\Big]=1-\int_\R r_q(y-x)\P_y\Big[\Gamma_{\bm{e}_q}^\mu\Big]\,\mu(dy).
\end{align}
Since $r_q$ is bounded and continuous, the map $x\mapsto \P_x[\Gamma_{\bm{e}_q}^\mu]$ is bounded, and $\mu$ has compact support, the dominated convergence theorem applied to this identity shows that $\varphi_q^\mu$ is continuous on $\R$. 

Multiplying both sides of equation (\ref{lem1-eq1}) by $r_q(0)$, we have
\begin{align*}
  \varphi_q^\mu(x)=r_q(0)\left(1-\int_\R  \varphi_q^\mu(y) \mu(dy)\right)+\int_\R  \varphi_q^\mu(y) h_q(x-y)\mu(dy).
     \stepcounter{equation}\tag{\theequation}
\end{align*}
Now, setting 
\begin{align}
  C_q^\mu:=r_q(0)\left(1-\int_\R \varphi_q^\mu(y) \mu(dy)\right),
\end{align}
we obtain equation (\ref{phiq-equation}). This completes the proof.
\end{proof}

We have thus obtained the equation satisfied by $\varphi_q^\mu$. To study the limit as $q\to0+$, we consider the following formal equations:
\begin{align}
\label{q-eq}
(I-T_q^\mu)f=C,&\qquad \mu[f]+\frac{C}{r_q(0)}=1,\\
\label{limit-eq}
(I-T^\mu)f=C,&\qquad \mu[f]=1,
\end{align}
where a function $f\in C(K)$ and a constant $C\in \R$ are unknown. Although these equations have only been introduced formally at this stage, in the remainder of this section we show that each of them has a unique solution. To obtain solutions to these equations, we use the general theory of compact operators. We first establish the following lemma.

\begin{lem}
\label{lem2.2}
$T_q^\mu$ and $T^\mu$ are compact operators on $C(K)$.
\end{lem}
\begin{proof}
Since the case of $T_q^\mu$ can be handled similarly, we only consider the case of $T^\mu$. Let $(f_n)\subset C(K)$ satisfy $\|f_n\|_{\max}\le M$ for some $M>0$. Since $h$ is continuous, we have
\begin{align}
|T^\mu f_n(x)|\le \int_K|f_n(y)|h(x-y)\,\mu(dy)\le M\mu(K)\max_{z\in K-K}h(z)<\infty,
 \end{align}
that is, the family $(T^\mu f_n)$ is uniformly bounded. Since $h$ is uniformly continuous on $K-K$ and
 \begin{align}
   \Big|T^\mu f_n(x)-T^\mu f_n(y)\Big|\le M\mu(K)\max_{z\in K}\Big|h(x-z)-h(y-z)\Big|,
 \end{align}
the family $(T^\mu f_n)$ is equicontinuous. Thus, by the Arzel\`{a}--Ascoli theorem, the family $(T^\mu f_n)$ is relatively compact in $C(K)$, that is, $T^\mu$ is a compact operator on $C(K)$. This completes the proof.
\end{proof}

We first consider equation (\ref{limit-eq}). To this end, we prove the following energy inequality.

\begin{lem}
\label{lem2.3}
Let $\nu$ be a finite signed measure with compact support $K$, and suppose that $\nu(\mathbb{R})=0$. Then we have
\begin{align}
\mathscr{E}_h(\nu):=\int_K\int_K h(x-y)\,\nu(dx)\nu(dy)\le 0.
\end{align}
\end{lem}
\begin{proof}
Since 
\begin{align}
\int_K\int_K e^{i\lambda(x-y)}\,\nu(dx)\nu(dy)&=\int_K e^{i\lambda x}\,\nu(dx)\int_K e^{-i\lambda y}\,\nu(dy)=|\widehat{\nu}(\lambda)|^2\ge 0,
\end{align}
taking the real part, we have
\begin{align}
\int_K\int_K \cos(\lambda(x-y))\,\nu(dx)\nu(dy)=|\widehat{\nu}(\lambda)|^2\ge 0.
\end{align}
Thus, using Lemma 3.3 (i) of Takeda--Yano \cite{Takeda-Yano} and the assumption $\nu(\R)=0$, we have
\begin{align*}
\mathscr{E}_h(\nu)&=\int_K \int_K h(x-y)\,\nu(dx)\nu(dy)\\
&=\frac{1}{2}\int_K\int_K \Big(h(x-y)+h(y-x)\Big)\,\nu(dx)\nu(dy)\\
&=\frac{1}{\pi}\int_K\int_K \left\{\int_0^\infty \Re \left(\frac{1-\cos (\lambda(x-y))}{\Psi(\lambda)}\right)\,d\lambda\right\}\,\nu(dx)\nu(dy)\\
&=\frac{1}{\pi}\int_0^\infty  \left\{\int_K\int_K \Big(1-\cos (\lambda(x-y))\Big)\,\nu(dx)\nu(dy)\right\}\Re \left(\frac{1}{\Psi(\lambda)}\right)\,d\lambda\\
&=-\frac{1}{\pi}\int_0^\infty |\widehat{\nu}(\lambda)|^2\frac{\Re \Psi(\lambda)}{|\Psi(\lambda)|^2}\,d\lambda\\
&\le 0,
\stepcounter{equation}\tag{\theequation}
\end{align*}
where the application of Fubini's theorem in the fourth equality is justified by Lemma 15.5 of Tsukada \cite{Tsukada} and the assumption that $\nu$ is a finite signed measure with compact support. This completes the proof.
\end{proof}

We distinguish two cases according to whether $1$ belongs to the spectrum of $T^\mu$. If $1$ does not belong to the spectrum, then $I-T^\mu$ is invertible, and equation (\ref{limit-eq}) can be solved easily:

\begin{prop}
Assume that $I-T^\mu$ is invertible; equivalently, $1$ does not belong to the spectrum of $T^\mu$. Then the unique solution to equation (\ref{limit-eq}) is
\begin{align}
  f=\frac{(I-T^\mu)^{-1}1}{\mu[(I-T^\mu)^{-1}1]},\qquad C=\frac{1}{\mu[(I-T^\mu)^{-1}1]}.
\end{align}
\end{prop}
\begin{proof}
Let $g:=(I-T^\mu)^{-1}1$. Since $I-T^\mu$ is invertible, it suffices to show that $\mu[g]\neq 0$. We assume that $\mu[g]=0$. We denote $g\cdot\mu:=g(x)\mu(dx)$. By Lemma \ref{lem2.3}, we have
\begin{align*}
0=\mu[g]&=\int_K g(x)\cdot (I-T^\mu)g(x)\,\mu(dx)\\
&=\int_K g^2(x)\,\mu(dx)-\int_K g(x)\,\mu(dx)\int_K g(y)h(x-y)\,\mu(dy)\\
&=\int_K g^2(x)\,\mu(dx)-\mathscr{E}_h(g\cdot \mu)\\
&\ge \int_K g^2(x)\,\mu(dx)\ge 0.
\stepcounter{equation}\tag{\theequation}
\end{align*}
Thus, we have $g=0,\ \mu$-a.e. Then we have
\begin{align}
1=(I-T^\mu)g(x)=g(x)-\int_K h(x-y)g(y)\,\mu(dy)=g(x)\qquad \text{for}\ x\in K.
\end{align}
This is a contradiction. This completes the proof.
\end{proof}

\begin{Rem}
By the Neumann series construction of the inverse operator, $I-T^\mu$ is invertible whenever the following condition holds:
\begin{align}
 \sup_{x\in K} \int_K h(x-y)\,\mu(dy)<1.
\end{align}
\end{Rem}

It remains to consider the case where $1$ belongs to the spectrum of $T^\mu$. In this case, since $I-T^\mu$ is not invertible, we appeal to the general theory of compact operators:

\begin{prop}
Assume that $1$ belongs to the spectrum of $T^\mu$. Then there exists a function $f^\mu\in C(K)$ such that 
\begin{align}
\label{prop2.6-eq1}
  \ker (I-T^\mu)=\mathrm{span}(f^\mu),\qquad \mu[f^\mu]=1.
\end{align}
Moreover, the unique solution to equation (\ref{limit-eq}) is
\begin{align}
f=f^\mu,\qquad C=0.
\end{align}
\end{prop}
\begin{proof}
Define the operator $\mu:C(K)\to \R$ by $\mu f:=\mu[f]$. Let $g_1\in \ker (I-T^\mu)\cap \ker (\mu)$. Since $g_1=T^\mu g_1$, Lemma \ref{lem2.3} implies
\begin{align}
  0\le \int_K g_1^2 (x)\,\mu(dx)=\int_K g_1(x)\cdot T^\mu g_1(x)\,\mu(dx)=\mathscr{E}_h(g_1\cdot \mu)\le 0.
\end{align}
Thus, we have $g_1=0,\ \mu$-a.e. Since $g_1$ is continuous on $K$ and $\mathrm{supp}(\mu)=K$, we have $g_1\equiv 0$, that is, the operator $\mu|_{\ker (I-T^\mu)}$ is injective. Thus, by the rank-nullity theorem, we have
\begin{align}
\dim \Big(\ker (I-T^\mu)\Big)=\dim \Big(\mathrm{Im}(\mu|_{\ker (I-T^\mu)})\Big)+\dim\Big(\ker(\mu|_{\ker (I-T^\mu)})\Big)\le 1.
\end{align}
Since $1$ belongs to the spectrum of $T^\mu$, we have $\dim (\ker (I-T^\mu))=1.$ Therefore, there exists a unique function $f^\mu\in C(K)$ satisfying (\ref{prop2.6-eq1}). 

We now prove the latter assertion. By the Fredholm alternative (see, e.g., Theorem 7.17 of \cite{Neerven} or Theorem 9.9 of \cite{FujitaKurodaIto}), we have
\begin{align}
\dim\Big(\ker (I-(T^\mu)^\ast)\Big)=\dim\Big(\ker (I-T^\mu)\Big)= 1.
\end{align}
Thus, there exists $\eta^\mu\in C(K)^\ast$ such that
\begin{align}
  \ker (I-(T^\mu)^\ast)=\mathrm{span}(\eta^\mu).
\end{align}
We will show that $\eta^\mu [1]\neq 0$. We assume that $\eta^\mu[1]=0.$ Since
\begin{align}
\Big\{f\in C(K):\ {}^\forall \eta\in \ker (I-(T^\mu)^\ast),\ \eta[f]=0\Big\}=\mathrm{Im}(I-T^\mu)
\end{align}
(see, e.g., the proof of Theorem 7.17 of \cite{Neerven} or Theorem 9.9 of \cite{FujitaKurodaIto}), there exists $g_2\in C(K)$ such that
\begin{align}
\label{prop2.6-eq2}
  (I-T^\mu)g_2=1.
\end{align}
Let $g_3:=g_2-\mu[g_2]f^\mu$. By (\ref{prop2.6-eq1}), we then have
\begin{align}
\mu[g_3]=\mu[g_2]-\mu[g_2]\mu[f^\mu]=0.
\end{align}
Since $(I-T^\mu)g_3=1$ by (\ref{prop2.6-eq1}) and (\ref{prop2.6-eq2}), Lemma \ref{lem2.3} implies
\begin{align*}
0&=\mu[g_3]=\int_K g_3(x)\cdot (I-T^\mu)g_3(x)\,\mu(dx)\\
&=\int_K g_3^2(x)\,\mu(dx)-\mathscr{E}_h(g_3\cdot \mu)\ge \int_K g_3^2(x)\,\mu(dx)\ge 0.
\stepcounter{equation}\tag{\theequation}
\end{align*}
Thus, we have $g_3\equiv 0$. However, this contradicts $(I-T^\mu)g_3=1$. Thus, we obtain $\eta^\mu[1]\neq 0.$

By the Fredholm alternative (see, e.g., Theorem 8.7-2 of \cite{Kreyszig} or Theorem 9.9 of \cite{FujitaKurodaIto}), equation (\ref{limit-eq}) has a solution if and only if $\eta^\mu[C]=0$, so we have $C=0$. In this case, since the solutions $f$ belong to $\ker (I-T^\mu)$, we can write $f=af^\mu$ for a constant $a\in \R$. But since $\mu[f]=1$, we have $a=1$. Thus, $f^\mu$ is the unique solution. This completes the proof.
\end{proof}

We have thus shown that equation (\ref{limit-eq}) admits a unique solution. An analogous result holds for equation (\ref{q-eq}). Since the proof follows essentially the same line of argument, with appropriate modifications such as replacing $h$ by $h_q$, we omit the details.

\begin{prop}
\begin{enumerate}
  \item Assume that $I-T_q^\mu$ is invertible. Then the unique solution to equation (\ref{q-eq}) is
 \begin{align}
   f=\frac{(I-T_q^\mu)^{-1}1}{\mu[(I-T_q^\mu)^{-1}1]+\frac{1}{r_q(0)}},\qquad C=\frac{1}{\mu[(I-T_q^\mu)^{-1}1]+\frac{1}{r_q(0)}}.
 \end{align}
 \item Assume that $1$ belongs to the spectrum of $T_q^\mu$. Then there exists a function $f_q^\mu \in C(K)$ such that
\begin{align}
  \ker (I-T_q^\mu)=\mathrm{span}(f_q^\mu),\qquad  \mu[f_q^\mu]=1.
\end{align}
Moreover, the unique solution to equation (\ref{q-eq}) is
\begin{align}
f=f_q^\mu,\qquad C=0.
\end{align}
\end{enumerate}
\end{prop}


\section{Kac Killing Penalization}
\label{S3}
For each $q>0$, let $(f_q^\mu,C_q^\mu)$ denote the unique solution to (\ref{q-eq}), and let $(f^\mu,C^\mu)$ denote the unique solution to (\ref{limit-eq}). Then the following relation holds:

\begin{prop}
\label{prop3.1}
The pair $(f_q^\mu,C_q^\mu)$ converges to $(f^\mu,C^\mu)$ as $q\to 0+$, that is, 
\begin{align}
   f_q^\mu\to f^\mu\ \text{in}\ C(K),\qquad C_q^\mu\to C^\mu.
 \end{align}
\end{prop}
\begin{proof}
First, we consider $q>0$ such that $C_q^\mu\ge 0$. Since $f_q^\mu=\varphi_q^\mu|_K\ge 0$ by the uniqueness of the solution, we have
 \begin{align}
   f_q^\mu=T_q^\mu f_q^\mu+C_q^\mu \ge C_q^\mu .
 \end{align}
Thus, by integrating with respect to $\mu$ and using (\ref{q-eq}), we have
 \begin{align}
    C_q^\mu \mu(K)\le \mu[f_q^\mu]=1-\frac{C_q^\mu}{r_q(0)}\le 1.
 \end{align}
 Therefore, if $C_q^\mu \ge 0$, then
 \begin{align}
   0\le C_q^\mu \le \frac{1}{\mu(K)}.
 \end{align}
Next, we consider $q>0$ such that $C_q^\mu<0$. Since the convergence $h_{q'} \to h$ is uniform on compact sets, we have
\begin{align}
  H:=\sup_{z\in K-K,\ 0<q'<q_0}h_{q'}(z)<\infty
\end{align}
for some fixed, sufficiently small $q_0>0$. For $0<q<q_0$ and $x\in K$, by (\ref{q-eq}), we have
 \begin{align*}
   0< -C_q^\mu&=T_q^\mu f_q^\mu(x)-f_q^\mu(x)\le T_q^\mu f_q^\mu(x)\\
   &=\int_K f_q^\mu(y)h_q(x-y)\,\mu(dy)\\
   &\le H\mu[f_q^\mu]= H\left(1-\frac{C_q^\mu}{r_q(0)}\right).
   \stepcounter{equation}\tag{\theequation}
 \end{align*}
Thus, we have
 \begin{align}
   (-C_q^\mu)\left(1-\frac{H}{r_q(0)}\right)\le H.
 \end{align}
By (\ref{rq-lim}), there exists $q_1>0$ such that $r_q(0)>2H$ for $0<q<q_1$. Therefore, if $C_q^\mu<0$, then
 \begin{align}
  0<-C_q^\mu \le 2H\qquad \text{for}\ 0<q<q_0\wedge q_1.
 \end{align}
 
Combining the above results, we have
 \begin{align}
   \sup_{0<q< q_0\wedge q_1}|C_q^\mu|\le \frac{1}{\mu(K)}\vee 2H<\infty.
 \end{align}
By (\ref{rq-lim}), we have
 \begin{align}
   \mu[f_q^\mu]=1-\frac{C_q^\mu}{r_q(0)}\to 1.
 \end{align}
Thus, there exists $q_2>0$ such that $|\mu[f_q^\mu]|\le 2$ for $0<q<q_2$. We define $q_3:=q_0\wedge q_1\wedge q_2.$ Since
 \begin{align}
   \|f_q^\mu\|_{\max}\le \|T_q^\mu f_q^\mu\|_{\max}+|C_q^\mu|\le H\mu[f_q^\mu]+|C_q^\mu|,
 \end{align}
$(f_q^\mu)_{0<q<q_3}\subset C(K)$ is uniformly bounded. Since
 \begin{align}
   |f_q^\mu(x)-f_q^\mu(y)|=|T_q^\mu f_q^\mu(x)-T_q^\mu f_q^\mu(y)|\le \int_K f_q^\mu(z)\Big|h_q(x-z)-h_q(y-z)\Big|\,\mu(dz)
 \end{align}
and $(h_q)_{q>0}$ is equicontinuous (see p. 14 of Takeda--Yano \cite{Takeda-Yano}), the equicontinuity of $(f_q^\mu)_{0<q<q_3}$ follows in the same way as in the proof of Lemma \ref{lem2.2}. Thus, by the Arzel\`{a}--Ascoli theorem, $(f_q^\mu)_{0<q<q_3}$ is relatively compact. 

For any subsequence $(q_n)$ such that $q_n\to 0+$, there exists a further subsequence $(q_{n_k})$, a function $f_\ast\in C(K)$, and a constant $C_\ast\in \R$ such that
 \begin{align}
   f_{q_{n_k}}^\mu\to f_\ast\ \text{in}\ C(K),\qquad C_{q_{n_k}}^\mu\to C_\ast.
 \end{align}
By uniform convergence and (\ref{q-eq}), we have
 \begin{align}
   \mu[f_\ast]=\lim_{k\to \infty}\mu[f_{q_{n_k}}^\mu]=\lim_{k\to \infty}\left(1-\frac{C_{q_{n_k}}^\mu}{r_{q_{n_k}}(0)}\right)=1.
 \end{align}
Since the convergence $h_{q_{n_k}}\to h$ is uniform on compact sets, we have $\|T_{q_{n_k}}^\mu -T^\mu\|\to 0.$ Thus, we have
 \begin{align}
   \|T_{q_{n_k}}^\mu f_{q_{n_k}}^\mu-T^\mu f_\ast \|_{\max} \le \|T_{q_{n_k}}^\mu -T^\mu\|\|f_{q_{n_k}}^\mu\|_{\max}+\|T^\mu\|\|f_{q_{n_k}}^\mu-f_\ast\|_{\max}\to 0.
 \end{align}
Therefore, the pair $(f_\ast,C_\ast)$ satisfies equation (\ref{limit-eq}). By the uniqueness of the solution to (\ref{limit-eq}), we have $(f_\ast,C_\ast)=(f^\mu,C^\mu)$. Since every subsequence has a further subsequence converging to the same limit $(f^\mu,C^\mu)$, the whole family $(f_q^\mu,C_q^\mu)$ converges to $(f^\mu,C^\mu)$. This completes the proof.
\end{proof}

We define
\begin{align}
\label{phimu-def}
  \varphi^\mu (x):=C^\mu +\int_K f^\mu(y)h(x-y)\,\mu(dy)\qquad \text{for}\ x\in \R.
\end{align}
This function is continuous.

Combining the preceding arguments, we obtain (\ref{tsuika-lim}).

\begin{prop}
\label{prop3.2}
  For every $x\in \R$,
  \begin{align}
  \label{prop3.2-eq1}
   \lim_{q\to 0+}r_q(0)\P_x\Big[\Gamma_{\bm{e}_q}^\mu\Big]=\varphi^\mu(x).
 \end{align}
This convergence is uniform on compact sets. In particular, $\varphi^\mu=f^\mu$ on $K.$
\end{prop}

The above convergence also holds in $L^1(\P_x)$:

\begin{prop}
For $x\in \R$ and $t\ge 0$, we have
\begin{align}
r_q(0)\P_{X_t}\Big[\Gamma_{\bm{e}_q}^\mu\Big]\to\varphi^\mu(X_t)\qquad \text{in}\ L^1(\P_x).
 \end{align}
\end{prop}
\begin{proof}
By (\ref{phiq-equation}) and (\ref{phimu-def}), we have
\begin{align*}
&\P_x\left[\Big|r_q(0)\P_{X_t}\Big[\Gamma_{\bm{e}_q}^\mu\Big]-\varphi^\mu(X_t)\Big|\right]\\
&\qquad \le |C_q^\mu-C^\mu|+\P_x\left[\left|\int_K f_q^\mu (y)h_q(X_t-y)\,\mu(dy)-\int_K f^\mu (y)h(X_t-y)\,\mu(dy)\right|\right]\\
&\qquad \le |C_q^\mu-C^\mu|+\|f_q^\mu-f^\mu\|_{\max}\int_K \P_x[h(X_t-y)]\,\mu(dy)\\
&\qquad \qquad+\Big(\|f_q^\mu-f^\mu\|_{\max}+\|f^\mu \|_{\max}\Big)\int_K \P_x\left[\Big|h_q(X_t-y)-h(X_t-y)\Big|\right]\,\mu(dy).
\stepcounter{equation}\tag{\theequation}
\end{align*}
By Proposition \ref{prop3.1} and Lemma 3.2 of Iba--Rivero \cite{Iba-Rivero}, the right-hand side converges to $0.$ This completes the proof.
\end{proof}

Combining the pointwise convergence with $L^1$-convergence, we obtain the following limit. We omit the proof, since similar proofs are given in previous studies of penalization problems (see, e.g., Theorems 4.2 and 4.4 of \cite{Takeda-Yano} or Theorem 3.4 of \cite{Iba-Yano}).

\begin{prop}
\label{prop3.4}
The process $(\varphi^\mu(X_t)\Gamma_t^\mu)_{t\ge 0}$ is a martingale. Moreover, we have
\begin{align}
   \lim_{q\to 0+}r_q(0)\P_x\Big[F_t\cdot \Gamma_{\bm{e}_q}^\mu\Big]=\P_x\Big[F_t\cdot \varphi^\mu(X_t)\Gamma_t^\mu\Big]
 \end{align}
 for any bounded $\F_t$-measurable functional $F_t$.
\end{prop}

The penalization limit (\ref{main-penallim}) follows immediately from this proposition.

We now prove the remaining assertions of Theorem \ref{mainthm1}. We first show that the function appearing in the limit is strictly positive.

\begin{prop}
The function $\varphi^\mu$ is strictly positive on $\R$.
\end{prop}
\begin{proof}
Since $r_q(0)\P_x[\Gamma_{\bm{e}_q}^\mu]\ge 0$, we have $\varphi^\mu(x)\ge 0$. Since
\begin{align}
1=\mu[f^\mu]=\int_K \varphi^\mu(y)\,\mu(dy),
\end{align}
there exists $x_0\in K$ such that $\varphi^\mu(x_0)>0.$ By recurrence, $T_{x_0}<\infty$, $\P_x$-a.s., for every $x\in\R$. We have
  \begin{align*}
  \P_{x}\Big[\Gamma_{\bm{e}_q}^\mu\Big]=\P_x\left[\int_0^\infty qe^{-qt}\Gamma_t^\mu\,dt\right]\ge \P_x\left[\int_{T_{x_0}}^\infty qe^{-qt}\Gamma_t^\mu \,dt\right]=\P_x\Big[e^{-qT_{x_0}}\Gamma_{T_{x_0}}^\mu\Big]\P_{x_0}\Big[\Gamma_{\bm{e}_q}^\mu\Big]
    \stepcounter{equation}\tag{\theequation}
  \end{align*}
  by the strong Markov property and the multiplicative property of $\Gamma_t^\mu$. Thus, by the dominated convergence theorem, we have
  \begin{align*}
    \varphi^\mu(x)\ge \lim_{q\to 0+}\P_x\Big[e^{-qT_{x_0}}\Gamma_{T_{x_0}}^\mu\Big]\cdot r_q(0)\P_{x_0}\Big[\Gamma_{\bm{e}_q}^\mu\Big]=\P_x\Big[\Gamma_{T_{x_0}}^\mu\Big]\varphi^\mu(x_0)>0.
    \stepcounter{equation}\tag{\theequation}
  \end{align*}
  This completes the proof.
\end{proof}

Finally, we consider the case where $\mu$ is a positive multiple of the equilibrium measure $\mu_K$.

\begin{prop}
\label{cor3.7}
Let $c> 0$. If $\mu=c\mu_K$, then
\begin{align}
   \varphi^{c\mu_K}(x)=\varphi_K(x)+\frac{1}{c}.
 \end{align}
\end{prop}
\begin{proof}
Since $\varphi_K(x)=0$ for $x\in K$, by (\ref{Iba-Rivero}), we have
\begin{align}
  T^{c\mu_K} 1(x)=c\int_K h(x-y)\,\mu_K(dy)=ck(K)\qquad \text{for}\ x\in K.
\end{align}
Thus, the pair $(\frac{1}{c},\frac{1}{c}-k(K))$ satisfies equation (\ref{limit-eq}). By the uniqueness of the solution, we have
\begin{align}
  \varphi^{c\mu_K}(x)=\frac{1}{c}\qquad \text{for}\ x\in K.
\end{align}
Therefore, by (\ref{Iba-Rivero}), we have
\begin{align}
\varphi^{c\mu_K}(x)=\frac{1}{c}-k(K)+\int_K \frac{1}{c}h(x-y)\,c\mu_K(dy)=\varphi_K(x)+\frac{1}{c}\qquad \text{for}\ x\in \R.
\end{align}
This completes the proof.
\end{proof}

This completes the proof of Theorem \ref{mainthm1}.


\section{Fredholm and Sturm--Liouville Equations}
\label{S4}
In this section, we prove Theorems \ref{mainthm2} and \ref{mainthm3}. In the course of the proof, we establish a number of lemmas. Since the proofs of some of them are lengthy and would make this section unnecessarily cumbersome, we defer them to Appendix \ref{Appendix}.

\subsection{Fredholm Equation Implies Sturm--Liouville Equation}
\label{S4.1}
In this subsection, we show that any function satisfying the Fredholm equation (\ref{main-Fredeq}) and the normalization condition (\ref{main-Frednorm}) also satisfies the Sturm--Liouville equation (\ref{main-SLeq}) and the growth condition (\ref{main-SLgrowth}).

We first state two basic lemmas concerning the infinitesimal generator of a \Levy\ process. Their proofs are deferred to Appendix \ref{Appendix}.

\begin{lem}
\label{lem4.1}
For $f\in C_c^\infty(\R)$, we have $\mathscr{L}f\in L^1(\R).$
\end{lem}

\begin{lem}
\label{lem4.2}
For $f\in C_c^\infty(\R)$, we have 
 \begin{align}
   \lim_{t\to 0+}\frac{P_tf-f}{t}= \mathscr{L} f\qquad \text{in}\ L^1(\R).
 \end{align}
\end{lem}

The following lemma shows that $\mathscr{L}h$ is meaningful in the weak sense. Although this lemma follows immediately from Lemma 5.3 of Berger--Schilling \cite{Berger-Schilling}, we include a proof for completeness.

\begin{lem}
\label{lem4.3}
  For $y\in \R$ and $f\in C_c^\infty(\R)$, we have
\begin{align}
   \int_\R h(x-y)|\widehat{\mathscr{L}}f(x)|\,dx<\infty.
\end{align}
Consequently, 
\begin{align}
\langle \mathscr{L} h(\cdot-y),f\rangle =\int_\R h(x-y)\widehat{\mathscr{L}}f(x)dx
\end{align}
is well-defined.
\end{lem}
\begin{proof}
Recall that $(\widehat{a},\widehat{\sigma}^2,\widehat{\Pi})$ denotes the \Levy\ triplet of the dual process. By (\ref{Lf-equation}), we have
\begin{align*}
  &h(x-y)|\widehat{\mathscr{L}}f(x)|\\
&\qquad \le  h(x-y)|\widehat{a}f'(x)|+\frac{\widehat{\sigma}^2}{2}h(x-y)|f''(x)|\\
&\qquad \qquad +\int_{|z|<1} h(x-y)\Big|f(x+z)-f(x)-z f'(x)\Big|\,\widehat{\Pi}(dz)\\
&\qquad \qquad +\int_{|z|\ge 1} h(x-y)\Big|f(x+z)-f(x)\Big|\,\widehat{\Pi}(dz).
\stepcounter{equation}\tag{\theequation}
\end{align*}
Since $f'$ and $f''$ have compact support and $h$ is continuous, the first and second terms of the right-hand side are integrable. By Taylor's theorem, the definition of the \Levy\ measure $\widehat{\Pi}$, and the compact support of $f''$, we have
\begin{align*}
 &\int_\R h(x-y)\,dx \int_{|z|<1}\Big|f(x+z)-f(x)-z f'(x)\Big|\,\widehat{\Pi}(dz)\\
  &\qquad \le \int_{|z|<1}z^2\,\widehat{\Pi}(dz)\int_0^1 (1-\theta)\,d\theta \int_\R h(x-y)|f''(x+\theta z)|\,dx<\infty.
  \stepcounter{equation}\tag{\theequation}
\end{align*}
We next consider the fourth term. By the subadditivity of $h$ and by $\widehat{\Pi}(-dz)=\Pi(dz)$, we have
\begin{align*}
\label{lem4.3-eq1}
  &\int_\R dx\int_{|z|\ge 1}h(x-y)\Big|f(x+z)-f(x)\Big|\,\widehat{\Pi}(dz)\\
  &\qquad \le \int_{|z|\ge 1}\,\widehat{\Pi}(dz)\int_\R h(x-y)|f(x)|\,dx+\int_{|z|\ge 1}\,\widehat{\Pi}(dz)\int_\R h(x-z-y)|f(x)|\,dx\\
  &\qquad \le 2\int_{|z|\ge 1}\,\widehat{\Pi}(dz)\int_\R h(x-y)|f(x)|\,dx+\int_{|z|\ge 1}h(z)\Pi(dz)\int_\R |f(x)|\,dx.
  \stepcounter{equation}\tag{\theequation}
\end{align*}
Since $f$ has compact support, the first term of the right-hand side is finite. Let $g(x):=1+h(x)$. Since 
\begin{align*}
\label{submultiplicative}
  g(x+y)\le 1+h(x)+h(y)\le 1+h(x)+h(y)+h(x)h(y)=g(x)g(y),
  \stepcounter{equation}\tag{\theequation}
\end{align*}
the function $g$ is sub-multiplicative. Thus, by Theorem 25.3 of \cite{Sato}, we have the following equivalence:
\begin{align}
\label{Sato-sub}
  \P_0[g(X_t)]<\infty\ \text{for}\ t\ge 0\qquad \Leftrightarrow\qquad \int_{|z|\ge 1}g(z)\,\Pi(dz)<\infty.
\end{align}
Since $\P_0[h(X_t)]<\infty$ (see Theorem 15.2 of Tsukada \cite{Tsukada}), the left-hand condition in this equivalence holds. Thus, the second term of the right-hand side of (\ref{lem4.3-eq1}) is finite. This completes the proof.
\end{proof}

Next, we show that the renormalized zero resolvent $h$ can be regarded as a fundamental solution of the extended operator $\mathscr{L}$ defined in (\ref{Def-exL}). In the case of stable processes, this has already been proved in Lemma 3.1 of Tsukada \cite{Tsukada-Stable}.

\begin{prop}
\label{prop4.5}
For $y\in \R$, we have
\begin{align}
\mathscr{L} h(\cdot-y)=\delta_y,
\end{align}
that is, for $f\in C_c^\infty(\R)$, we have
\begin{align}
\int_\R h(x-y)\widehat{\mathscr{L}}f(x)\,dx=f(y).
\end{align}
\end{prop}
\begin{proof}
We denote by $(\widehat P_t)_{t\ge0}$ and $\widehat U^q$ the transition semigroup and the $q$-potential of the dual process $\widehat{X}$, respectively. Since Lebesgue measure $dx$ is invariant for the \Levy\ process (see, e.g., Theorem 24.24 of \cite{Sato}), by Lemma \ref{lem4.2}, we have
\begin{align}
\label{Lf-int0}
\int_\R \widehat{\mathscr{L}}f(x)\,dx&=\lim_{t\to 0+}\frac{1}{t}\int_\R \Big(\widehat{P}_t f(x)-f(x)\Big)\,dx=\lim_{t\to 0+}\frac{1}{t}\int_\R \Big(f(x)-f(x)\Big)\,dx=0.
\end{align}
By (\ref{qI-L}), we have
\begin{align*}
   \int_\R h_q(x-y)\widehat{\mathscr{L}} f(x)\,dx&=r_q(0)\int_\R \widehat{\mathscr{L}} f(x)\,dx-\int_{\R} r_q(y-x)\widehat{\mathscr{L}} f(x)\,dx\\
  &=-\widehat{U}^q\widehat{\mathscr{L}}f(y)\\
  &=-q\widehat{U}^qf(y)+f(y)\\
  &=-\int_\R qr_q(y-x)f(x)\,dx+f(y).
  \stepcounter{equation}\tag{\theequation}
\end{align*}
By (\ref{rq-lim}), we have
\begin{align*}
\left|\int_\R qr_q(y-x)f(x)\,dx\right|&\le \int_\R qr_q(y-x)|f(x)|\,dx\le  qr_q(0)\int_\R |f(x)|\,dx\to 0.
\stepcounter{equation}\tag{\theequation}
\end{align*}
By (\ref{hq-def}), we have
\begin{align*}
|h_q(x-y)\widehat{\mathscr{L}}f(x)|&\le \Big(h_q(x-y)+h_q(y-x)\Big)|\widehat{\mathscr{L}}f(x)|\\
&\le\frac{2}{\pi}\int_0^\infty \left|\Re \left(\frac{1-\cos ((x-y)\lambda)}{q+\Psi(\lambda)}\right)\right|\,d\lambda\cdot |\widehat{\mathscr{L}}f(x)|\\
&\le \frac{2}{\pi}\int_0^\infty  \left|\frac{1-\cos ((x-y)\lambda)}{q+\Psi(\lambda)}\right|\,d\lambda\cdot |\widehat{\mathscr{L}}f(x)|\\
&\le \frac{2}{\pi}\int_0^\infty \frac{1-\cos ((x-y)\lambda)}{|\Psi(\lambda)|}\,d\lambda\cdot |\widehat{\mathscr{L}}f(x)|\\
&=:H(x-y)|\widehat{\mathscr{L}}f(x)|.
\stepcounter{equation}\tag{\theequation}
\end{align*}
Here, by Lemma 15.5 of Tsukada \cite{Tsukada}, the function $H$ is finite and continuous. Since 
\begin{align*}
1-\cos (a+b)&=2\sin^2 \left(\frac{a+b}{2}\right)=2\left(\sin \frac{a}{2}\cos \frac{b}{2}+\cos \frac{a}{2}\sin \frac{b}{2}\right)^2\\
&\le 4\sin^2 \frac{a}{2}\cos^2 \frac{b}{2}+4\cos^2 \frac{a}{2}\sin^2 \frac{b}{2}\\
&\le 4\sin^2 \frac{a}{2}+4\sin^2 \frac{b}{2}\\
&\le 2(1-\cos a)+2(1-\cos b),
\stepcounter{equation}\tag{\theequation}
\end{align*}
we have
\begin{align}
H(a+b)\le 2H(a)+2H(b).
\end{align}
Let $G(z):=2+H(z)$. By the same argument as that used for (\ref{submultiplicative}), $G$ is sub-multiplicative. By the proof of Theorem 15.2 of \cite{Tsukada}, we have $\P_0[H(X_t)]<\infty$. By the same argument as in the proof of Lemma \ref{lem4.3} ( or Lemma 5.3 of Berger-Schilling \cite{Berger-Schilling}), we have $H(x-y)|\widehat{\mathscr{L}}f(x)|\in L^1(\R)$. Therefore, by the dominated convergence theorem, we obtain
\begin{align*}
  \int_\R h(x-y)\widehat{\mathscr{L}} f(x)\,dx&=\lim_{q\to 0+}\int_\R h_q(x-y)\widehat{\mathscr{L}} f(x)\,dx\\
  &=-\lim_{q\to 0+}\int_\R qr_q(y-x)f(x)\,dx+f(y)=f(y).
  \stepcounter{equation}\tag{\theequation}
\end{align*}
This completes the proof.
\end{proof}

We now prove that assertion 1 implies assertion 2 in Theorem \ref{mainthm2}.

\begin{prop}
Assume that the pair $(C,f)$ satisfies the Fredholm equation (\ref{main-Fredeq}) with the normalization condition (\ref{main-Frednorm}). Then the function $f$ satisfies the Sturm--Liouville equation (\ref{main-SLeq}).
\end{prop}
\begin{proof}
Let $g\in C_c^\infty(\R)$. By (\ref{main-Fredeq}), (\ref{Lf-int0}), and Proposition \ref{prop4.5}, we have
\begin{align*}
\langle \mathscr{L}f,g\rangle &=\int_\R\left(C +\int_K f(y)h(x-y)\,\mu(dy)\right)\widehat{\mathscr{L}}g(x)\,dx\\
&=C \int_\R \widehat{\mathscr{L}}g(x)\,dx+\int_K f (y)\,\mu(dy)\int_\R h(x-y)\widehat{\mathscr{L}}g(x)\,dx\\
&=\int_K f(y)g(y)\,\mu(dy).
\stepcounter{equation}\tag{\theequation}
\end{align*}
This is the desired equation (\ref{main-SLeq}).
\end{proof}

Next, we verify the growth condition (\ref{main-SLgrowth}).

\begin{prop}
\label{prop4.7}
Assume that the pair $(C,f)$ satisfies the Fredholm equation (\ref{main-Fredeq}) with the normalization condition (\ref{main-Frednorm}). Then the function $f$ satisfies the growth condition (\ref{main-SLgrowth}). More precisely, we have
\begin{align}
\label{prop4.6-eq1}
f(x)= h(x)+C\mp \frac{1}{\P_0[X_1^2]}\int_K yf (y)\,\mu(dy)+o(1)\qquad \text{as}\ x\to \pm \infty.
\end{align}
\end{prop}
\begin{proof}
By (\ref{h-limit2}), we have
  \begin{align*}
  \lim_{x\to \pm \infty}\Big(f(x)-h(x)\Big)&=C+ \lim_{x\to \pm \infty}\int_K \Big\{h(x-y)-h(x)\Big\}f (y)\,\mu(dy)\\
  &=C\mp \frac{1}{\P_0[X_1^2]}\int_K yf (y)\,\mu(dy)\in \R.
  \stepcounter{equation}\tag{\theequation}
  \end{align*}
This completes the proof.
\end{proof}

We have thus proved one direction of the equivalence asserted in Theorem \ref{mainthm2}.

\subsection{Sturm--Liouville Equation Implies Fredholm Equation}
\label{S4.2}

In this subsection, we prove the converse of the result established in the previous subsection. More precisely, we show that any function satisfying the Sturm--Liouville equation (\ref{main-SLeq}) and the growth condition (\ref{main-SLgrowth}) also satisfies the Fredholm equation (\ref{main-Fredeq}) and the normalization condition (\ref{main-Frednorm}).

We first state a basic lemma concerning the infinitesimal generator of a \Levy\ process. Its proof is deferred to Appendix \ref{Appendix}.

\begin{lem}
\label{lem4.8}
For $f\in C_b(\R)$, if $\mathscr{L}f=c\delta_0$, then $c=0.$
\end{lem}

We now prove that assertion 2 implies assertion 1 in Theorem \ref{mainthm2}.

\begin{prop}
\label{prop4.9}
Assume that a function $f$ satisfies the Sturm--Liouville equation (\ref{main-SLeq}) with the growth condition (\ref{main-SLgrowth}). Then the function $f$ satisfies the normalization condition (\ref{main-Frednorm}).
\end{prop}
\begin{proof}
Let
\begin{align}
  F(x):=\int_K f(y)h(x-y)\,\mu(dy)\qquad \text{for}\ x\in \R.
\end{align}
For $g\in C_c^\infty(\R)$, by Proposition \ref{prop4.5}, we have
\begin{align*}
  \langle \mathscr{L}F,g\rangle&=\int_\R \widehat{\mathscr{L}}g(x)\,dx\int_K f(y)h(x-y)\,\mu(dy)\\
  &=\int_K f(y)\,\mu(dy)\int_\R h(x-y)\widehat{\mathscr{L}}g(x)\,dx\\
  &=\int_K f(y)g(y)\,\mu(dy)\\
  &=\langle f\cdot \mu,g\rangle.
  \stepcounter{equation}\tag{\theequation}
\end{align*}
Thus, by (\ref{main-SLeq}), we have
\begin{align}
\label{prop4.9-eq1}
\mathscr{L}(f-F)=0.
\end{align}
We have
\begin{align*}
  G(x)&:=(f(x)-F(x))-(1-\mu[f])h(x)\\
  &=(f(x)-h(x))-\Big(F(x)-\mu[f]h(x)\Big)\\
  &=(f(x)-h(x))-\int_K f (y)\Big(h(x-y)-h(x)\Big)\,\mu(dy).
  \stepcounter{equation}\tag{\theequation}
\end{align*}
Using the subadditivity of $h$ and (\ref{main-SLgrowth}), we have $G\in C_b(\R).$ Moreover, by Proposition \ref{prop4.5}, we have 
\begin{align}
  \mathscr{L}G=\mathscr{L}(f-F)-(1-\mu[f])\mathscr{L}h=-(1-\mu[f])\delta_0.
\end{align}
Therefore, by Lemma \ref{lem4.8}, we obtain $\mu[f]=1.$ This completes the proof.
\end{proof}

Next, we verify the Fredholm equation (\ref{main-Fredeq}).

\begin{prop}
Assume that a function $f$ satisfies the Sturm--Liouville equation (\ref{main-SLeq}) with the growth condition (\ref{main-SLgrowth}). Then the function $f$ satisfies the Fredholm equation (\ref{main-Fredeq}).
\end{prop}
\begin{proof}
By Proposition \ref{prop4.9} and its proof, we have $f-F\in C_b(\R)$. By (\ref{Psi-zero}), (\ref{prop4.9-eq1}), and the Liouville theorem for \Levy\ processes (see Theorem 4.4 of Berger--Schilling \cite{Berger-Schilling}), there exists a constant $a\in \R$ such that
\begin{align}
f(x)-F(x)=a\qquad \text{for a.e.}\ x\in \R.
\end{align}
Since $f-F$ is continuous, this equation holds for all $x\in \R$. This is the desired equation.
\end{proof}

This completes the proof of Theorem \ref{mainthm2}.

\subsection{The Growth Condition in the Finite Second Moment Case}
\label{S4.3}

In this subsection, we will prove Theorem \ref{mainthm3}.

We first state a basic lemma concerning the infinitesimal generator of a \Levy\ process. Its proof is deferred to Appendix \ref{Appendix}.

\begin{lem}
\label{lem4.12}
Assume that $\P_0[X_1^2]<\infty$. If a function $f\in \mathscr{A}$ satisfies
\begin{align}
   \mathscr{L}f=0,\qquad |f(x)|\le C(1+|x|),
 \end{align}
then there exist constants $a,b\in \R$ such that
 \begin{align}
   f(x)=ax+b\qquad \text{for\ a.e.}\ x\in \R.
 \end{align}
\end{lem}

We now prove Theorem \ref{mainthm3}.

\begin{proof}[Proof of Theorem \ref{mainthm3}]
Assume that assertion i holds. Then, by Theorem \ref{mainthm2}, the Fredholm equation (\ref{main-Fredeq}) holds, and hence assertion ii follows from Proposition \ref{prop4.7} and (\ref{h-limit2}).

Conversely, we assume that assertion ii holds. Let $G(x):=f(x)-F(x)$. By (\ref{h-limit2}), we have
\begin{align*}
\label{lem4.12-eq3}
  G(x+z)-G(x)&=\Big(f(x+z)-f(x)\Big)-\int_Kf(y)\Big(h(x+z-y)-h(x-y)\Big)\,\mu(dy)\\
  &\to \pm \frac{(1-\mu[f])z}{\P_0[X_1^2]}\qquad \text{as}\ x\to \pm \infty.
  \stepcounter{equation}\tag{\theequation}
\end{align*}
Thus, there exists a constant $M>1$ such that
\begin{align}
  |G(x+1)-G(x)|\le M\qquad \text{for}\ x\in \R.
\end{align}
For $x\ge 0$, choose $n_x\in \N\cup\{0\}$ such that $n_x\le x<n_x+1$. Let $r:=x-n_x$. Thus, we have
\begin{align*}
  |G(x)|&\le |G(r)|+\sum_{k=1}^{n_x} \Big|G(r+k)-G(r+k-1)\Big|\\
  &\le \max_{r\in [0,1]}|G(r)|+Mn_x\\
  &\le \max_{r\in [0,1]}|G(r)|+Mx\qquad \text{for}\ x\ge 0.
    \stepcounter{equation}\tag{\theequation}
\end{align*}
Since the same conclusion holds for $x<0$, there exists a constant $C>0$ such that
\begin{align}
|G(x)|\le C(1+|x|)\qquad \text{for}\ x\in \R.
\end{align}
Thus, by (\ref{prop4.9-eq1}) and Lemma \ref{lem4.12}, there exist constants $a,b\in \R$ such that
\begin{align}
G(x)=ax+b\qquad \text{for a.e.}\ x\in \R.
\end{align}
Since $G$ is continuous, this equality holds for all $x\in \R$. We have
\begin{align}
  G(x+z)-G(x)=az\qquad \text{for}\ x\in \R.
\end{align}
Comparing (\ref{lem4.12-eq3}) as $x\to \pm \infty$, we obtain $a=0$ and $\mu[f]=1$. Thus, we have
\begin{align*}
 f(x)-h(x)&=\Big(f(x)-F(x)\Big)+\Big(F(x)-h(x)\Big)\\
  &=b+\int_K f(y)\Big(h(x-y)-h(x)\Big)\,\mu(dy)\\
  &\to b\mp\frac{1}{\P_0[X_1^2]}\int_Kyf(y)\,\mu(dy)\in \R. 
  \stepcounter{equation}\tag{\theequation}
\end{align*}
This implies $f-h\in C_b(\R).$ This completes the proof.
\end{proof}


\section{Penalized Process}
\label{S5}
In this section, we investigate properties of the penalized process obtained from the penalization limit (\ref{main-penallim}).

For $x\in \R$, recall that $\Q_x^\mu$ denotes the penalized law, that is, the measure $\Q_x^\mu$ on $\F_\infty$ satisfies
\begin{align}
\frac{d\Q_x^\mu}{d\P_x}\Big|_{\F_t}=\frac{\varphi^\mu(X_t)\cdot \Gamma_t^\mu}{\varphi^\mu(x)}\qquad \text{for}\ t\ge 0.
\end{align}
By an argument analogous to the proof of Theorem 1.4 of Takeda \cite{T}, we obtain
\begin{align}
|X_t|\to\infty,
\qquad
\mathbb{Q}_x^\mu\text{-a.s.}
\end{align}
Thus, the process is transient under $\mathbb{Q}_x^\mu$. Since the original process is recurrent, it follows that
\begin{align}
\mathbb{P}_x\sim\mathbb{Q}_x^\mu
\quad\text{on }\F_t,
\qquad
\mathbb{P}_x\perp\mathbb{Q}_x^\mu
\quad\text{on }\F_\infty.
\end{align}

Furthermore, the $\mu$-weighted total local time accumulated on $K$ is finite.

\begin{prop}
For every $x\in \R$, we have
  \begin{align}
    \int_K L_\infty^y\,\mu(dy)<\infty,\qquad \Q_x^\mu\text{-a.s.}
  \end{align}
\end{prop}
\begin{proof}
Let $(\tau_u)_{u\ge 0}$ denote the right-continuous inverse of the additive functional $(\int_{K} L_t^y\,\mu(dy))_{t\geq 0}$, that is,
\begin{align}
\tau_u:=\inf \left\{t\ge 0:\ \int_K L_t^y\, \mu(dy)>u\right\}\qquad \text{for}\ u\ge 0.
\end{align}
By recurrence, we have $\int_K L_\infty^y\,\mu(dy)=\infty,\ \mathbb{P}_x\text{-a.s.}$, and hence, for every $u\geq 0$, $\tau_u<\infty,\ \mathbb{P}_x\text{-a.s.}$ Since $\tau_u$ is a stopping time, the optional sampling theorem yields
  \begin{align*}
    \Q_x^\mu\left(\int_K L_t^y\,\mu(dy)>u\right)&=\Q_x^\mu (\tau_u< t)\\
    &=\frac{1}{\varphi^\mu(x)}\P_x\left[1_{\{\tau_u< t\}}\varphi^\mu(X_t)\cdot \Gamma_t^\mu\right]\\
    &=\frac{1}{\varphi^\mu(x)}\P_x\left[1_{\{\tau_u<t\}}\varphi^\mu(X_{\tau_u})\cdot \Gamma_{\tau_u}^\mu\right]\\
    &=\frac{e^{-u}}{\varphi^\mu(x)}\P_x\left[1_{\{\tau_u< t\}}\varphi^\mu(X_{\tau_u})\right].
    \stepcounter{equation}\tag{\theequation}
  \end{align*}
Thus, letting $t\to \infty$, we have
  \begin{align}
    \Q_x^\mu\left(\int_K L_\infty^y\,\mu(dy)>u\right)=\frac{e^{-u}}{\varphi^\mu(x)}\P_x\left[\varphi^\mu(X_{\tau_u})\right].
  \end{align}
Since $X_{\tau_u}\in K$ and $\varphi^\mu$ is continuous, we have
  \begin{align}
    \frac{\min_{z\in K}\varphi^\mu(z)}{\varphi^\mu(x)}\cdot e^{-u}\le \Q_x^\mu\left(\int_K L_\infty^y \,\mu(dy)>u\right)\le \frac{\max_{z\in K}\varphi^\mu(z)}{\varphi^\mu(x)}\cdot e^{-u}.
  \end{align}
Letting $u\to\infty$, we obtain the desired conclusion.
\end{proof}

We define the measure obtained by removing the weight as follows:
\begin{align}
  \mathscr{P}_x^\mu:=\frac{\varphi^\mu(x)}{\Gamma_\infty^\mu}\cdot \Q_x^\mu\qquad \text{on}\ \F_\infty.
\end{align}
Since
\begin{align}
\mathscr{P}_x^\mu \left(\Gamma_\infty^\mu>\frac{1}{n}\right)\le n\varphi^\mu(x)<\infty\qquad \text{for}\ n\ge 1,
\end{align}
the measure $\mathscr{P}_x^\mu$ is a $\sigma$-finite measure. It is clear that $\mathscr{P}_x^\mu[\Gamma_\infty^\mu]=\varphi^\mu(x)$, and hence, we have
\begin{align}
\Q_x^\mu=\frac{\Gamma_\infty^\mu}{\mathscr{P}_x^\mu[\Gamma_\infty^\mu]}\cdot \mathscr{P}_x^\mu\qquad \text{on}\ \F_\infty.
\end{align}

Assume that $X$ admits transition densities $p_u(\cdot)$ for $u>0$. It was shown in Subsection 8.2 of Takeda--Yano \cite{Takeda-Yano} that, for the penalization problem with $\mu=\delta_0$, the measure obtained by removing the weight $\mathscr{P}_x$ is given as follows:
\begin{align}
\mathscr{P}_x:=\int_0^\infty (\P_{x,0}^u\bullet \P_0^{(0)})p_u(-x)\,du+h(x)\cdot\P_x^{(0)},
\end{align}
where $\mathbb{P}_{x,0}^u$ denotes the \Levy\ bridge from $X_0=x$ to $X_u=0$, $\mathbb{P}_x^{(0)}$ denotes the law of the \Levy\ process conditioned to avoid zero, and the symbol $\bullet$ denotes concatenation.

In this case, $\mathscr{P}_x^\mu$ coincides with $\mathscr{P}_x$. Therefore, $\mathscr{P}_x^\mu$ is independent of the weight $\Gamma^\mu$. In this sense, the measure $\mathscr{P}_x$ is universal.

\begin{prop}
For $x\in \R$, we have
\begin{align}
  \mathscr{P}_x^\mu=\mathscr{P}_x.
\end{align}
\end{prop}
\begin{proof}
Note that $\varphi^{\delta_0}(x)=1+h(x)$. By transience, we have
\begin{align}
  \Gamma_t^\mu\to \Gamma_\infty^\mu>0,\ \Q_x^{\delta_0}\text{-a.s.},\qquad \Gamma_t^{\delta_0}\to \Gamma_\infty^{\delta_0}>0,\ \Q_x^{\mu}\text{-a.s.}
\end{align}
By Proposition 2.1 of Yano \cite{Yano-penal}, we have
\begin{align}
\varphi^{\mu}(X_t)\to \infty,\ \Q_x^\mu\text{-a.s.},\qquad h(X_t)\to \infty,\ \Q_x^{\delta_0}\text{-a.s.}
\end{align}
Since $\varphi^\mu-h$ is bounded, we also have
\begin{align}
h(X_t)\to \infty,\qquad \Q_x^{\mu}\text{-a.s.}
\end{align}
Thus, we have
\begin{align}
  \lim_{t\to \infty}\frac{\varphi^{\mu}(X_t)}{\varphi^{\delta_0}(X_t)}=\lim_{t\to \infty}\frac{\varphi^{\mu}(X_t)}{1+h(X_t)}=1,\qquad \Q_x^{\delta_0}\text{-a.s.}\ \text{and}\ \Q_x^\mu\text{-a.s.}
\end{align}
Therefore, Theorem 4.1 of Yano \cite{Yano-penal} applies, and the assertion follows.
\end{proof}

Using the fact that $(\Gamma_t^\mu)_{t\ge 0}$ is a multiplicative functional, we obtain Theorem \ref{main-thm4} and the following corollary:

\begin{cor}
Let $g:=\sup\{t\ge 0:\ X_t=0\}$ denote the last exit time from zero, where we use the convention that $\sup\emptyset=0$. Its law is given by
\begin{align}
  \Q_x^\mu(g\in du)=\frac{\P_{x,0}^u[\Gamma_u^\mu]\P_0^{(0)}[\Gamma_\infty^\mu]p_u(-x)}{\varphi^\mu(x)}1_{\{u>0\}}du+\frac{h(x)\P_x^{(0)}[\Gamma_\infty^\mu]}{\varphi^\mu(x)}\delta_0(du).
\end{align}
\end{cor}


\appendix
\section{Appendix: Proofs of Auxiliary Lemmas}
\label{Appendix}
In this appendix, we provide the proofs omitted from the main text.

\begin{proof}[Proof of (\ref{Psi-zero})]
Assume that there exists $\lambda_0>0$ such that $\Psi(\lambda_0)=0$. By a well-known property of characteristic functions (see, e.g., Theorem 3.5.2 of \cite{Durrett}), $\Psi$ is periodic, that is,
\begin{align}
  \Psi(\lambda+\lambda_0)=\Psi(\lambda)\qquad \text{for}\ \lambda\in \R.
\end{align}
Thus, we have
\begin{align*}
  \int_0^{\infty} \left|\frac{1}{q+\Psi(\lambda)}\right|d\lambda&=\sum_{n=0}^\infty\int_{n\lambda_0}^{(n+1)\lambda_0}\left|\frac{1}{q+\Psi(\lambda)}\right|d\lambda=\sum_{n=0}^\infty\int_{0}^{\lambda_0}\left|\frac{1}{q+\Psi(\lambda)}\right|d\lambda.
\end{align*}
Since $\Re \Psi(\lambda)\ge 0$ and $\Psi$ is continuous, the right-hand side equals infinity. This contradicts Hypothesis \ref{Hyp}. This completes the proof.
\end{proof}

\begin{proof}[Proof of Lemma \ref{lem4.1}]
Since $f\in C_c^\infty(\R)$, the first and second terms of the right-hand side are integrable. By Taylor's theorem and by the definition of a \Levy\ measure $\Pi$, we have
\begin{align*}
 &\int_\R \,dx \int_{|y|<1}\Big|f(x+y)-f(x)-y f'(x)\Big|\,\Pi(dy)\\
  &\qquad \le \int_{|y|<1}y^2\,\Pi(dy)\int_0^1 (1-\theta)\,d\theta \int_\R |f''(x+\theta y)|\,dx\\
  &\qquad =\frac{\|f''\|_{L^1}}{2}\int_{|y|<1}y^2\,\Pi(dy)<\infty.
  \stepcounter{equation}\tag{\theequation}
\end{align*}
By the definition of a \Levy\ measure $\Pi$, we have
\begin{align*}
  \int_\R \,dx \int_{|y|\ge 1}\Big|f(x+y)-f(x)\Big|\,\Pi(dy)&=\int_{|y|\ge 1}\,\Pi(dy)\int_\R \Big|f(x+y)-f(x)\Big|\,dx\\
  &\le 2\|f\|_{L^1}\int_{|y|\ge 1}\,\Pi(dy)<\infty.
  \stepcounter{equation}\tag{\theequation}
\end{align*}
Thus, the third term of the right-hand side of (\ref{Lf-equation}) is integrable. This completes the proof.
\end{proof}

\begin{proof}[Proof of Lemma \ref{lem4.2}]
By Dynkin's formula (see, e.g., Lemma 19.21 of \cite{Kallenberg}), the process
\begin{align}
\left(f(X_t)-f(x)-\int_0^t \mathscr{L}f(X_s)\,ds\right)_{t\ge 0}
\end{align}
is a $\P_x$-martingale. Thus, we have
\begin{align}
  P_t f(x)-f(x)=\P_x\left[\int_0^t \mathscr{L}f(X_s)\,ds\right]=\int_0^t P_s \mathscr{L}f(x)\,ds.
\end{align}
Since the semigroup $(P_t)$ restricted to $C_0(\R)$ extends to a strongly continuous semigroup on $L^1(\R)$ (see, e.g., Proposition 12.7 of \cite{Berg-Forst}), we have
\begin{align}
  \left\|\frac{P_t f-f}{t}-\mathscr{L}f\right\|_{L^1}&\le \frac{1}{t}\int_0^t \Big\|P_s\mathscr{L}f-\mathscr{L}f\Big\|_{L^1}\,ds\to 0.
\end{align}
This completes the proof.
\end{proof}

\begin{proof}[Proof of Lemma \ref{lem4.8}]
Let $u:=ch-f$. By Proposition \ref{prop4.5}, we have
\begin{align}
\label{lem4.8-eq1}
\mathscr{L}u=c\mathscr{L}h-\mathscr{L}f=0.
\end{align}
For $z\in \R$, we define
\begin{align}
\Delta_z u(x):=u(x+z)-u(x).
\end{align}
By the subadditivity of $h$, we have
\begin{align*}
\label{lem4.8-eq2}
|\Delta_zu(x)|&\le |c||h(x+z)-h(x)|+|f(x)-f(x+z)|\\
&\le |c|\max \{h(z),h(-z)\}+2\|f\|_{\max}<\infty.
\stepcounter{equation}\tag{\theequation}
\end{align*}
Since $g(\cdot-z)-g\in C_c^\infty(\R)$ for $g\in C_c^\infty(\R)$, by (\ref{lem4.8-eq1}), we have
\begin{align*}
\langle \mathscr{L}(\Delta_z u),g\rangle &=\int_\R \Delta_zu(x)\widehat{\mathscr{L}}g(x)\,dx\\
&=\int_\R u(x+z)\widehat{\mathscr{L}}g(x)\,dx-\int_{\R}u(x)\widehat{\mathscr{L}}g(x)\,dx\\
&=\int_\R u(x)\widehat{\mathscr{L}}g(x-z)\,dx-\int_{\R}u(x)\widehat{\mathscr{L}}g(x)\,dx\\
&=\langle \mathscr{L} u,g(\cdot-z) -g\rangle\\
&=0.
\stepcounter{equation}\tag{\theequation}
\end{align*}
Thus, by (\ref{Psi-zero}) and the Liouville theorem for \Levy\ processes (see Theorem 4.4 of Berger--Schilling \cite{Berger-Schilling}), there exists a constant $A_z$ such that
\begin{align}
\Delta_zu\equiv A_z\qquad \text{for}\ x\in \R.
\end{align}
For $z,w\in \R$, we have
\begin{align*}
A_{z+w}&=u(x+z+w)-u(x)\\
&=\Big(u(x+z+w)-u(x+z)\Big)+\Big(u(x+z)-u(x)\Big)\\
&=A_w+A_z.
\stepcounter{equation}\tag{\theequation}
\end{align*}
Since $h$ is continuous, by (\ref{lem4.8-eq2}), the function $z\mapsto A_z$ is locally bounded. The standard regularity theorem for Cauchy's functional equation implies that there exists $a\in \R$ such that
\begin{align}
A_z=az\qquad \text{for}\ z\in \R
\end{align}
(see, e.g., Corollary 1.6.12 and Theorem 1.6.11 of \cite{B}). Thus, we have
\begin{align}
  u(x+z)-a(x+z)=u(x)-ax+\Delta_zu(x)-az=u(x)-ax.
\end{align}
Hence, there exists a constant $b\in \R$ such that
\begin{align}
  u(x)=ax+b\qquad \text{for}\ x\in \R.
\end{align}
Therefore, we have
\begin{align}
f(x)=ch(x)-ax-b\qquad \text{for}\ x\in \R,
\end{align}
and so
\begin{align}
\label{lem4.8-eq4}
f(x)+f(-x)=c(h(x)+h(-x))-2b.
\end{align}
Since $h(x)+h(-x)\to \infty$ as $x\to \pm \infty$ (see Lemma 3.10 of Pant\'{i} \cite{Panti}) while $f$ is bounded, (\ref{lem4.8-eq4}) implies $c=0$. Thus, we have $f(x)=-ax-b$, and the boundedness of $f$ implies $a=0$. This completes the proof.
\end{proof}

\begin{proof}[Proof of Lemma \ref{lem4.12}]
By Theorem 25.3 of \cite{Sato}, the assumption $\P_0[X_1^2]<\infty$ implies
\begin{align}
  \int_{|z|\ge 1}|z|\,\Pi(dz)\le \int_{|z|\ge 1}z^2\,\Pi(dz)<\infty.
\end{align}
Thus, by Example 5.11 of Grzywny–Kwaśnicki \cite{Grzywny-Kwasnicki} with $\alpha=1$, the function $f$ agrees almost everywhere with a polynomial. By assumption, the function $f$ must be a polynomial of degree at most one. This completes the proof.
\end{proof}


\section{Appendix: Penalization Via the Constant Clock}
\label{Appendix2}
In this appendix, we impose the additional assumption that the function $q\mapsto r_q(0)$ is regularly varying at $0+$ with index $-\rho$ for some $\rho\in (0,1)$, that is,
\begin{align}
\label{RV-rq}
\lim_{q\to 0+}\frac{r_{q\theta}(0)}{r_q(0)}=\theta^{-\rho}\qquad \text{for}\ \theta>0.
\end{align}
For standard Brownian motion, this condition holds with $\rho=1/2$. For a recurrent symmetric $\alpha$-stable process, it holds with $\rho=1-1/\alpha$.

Under this additional assumption, we show that the penalization limit in (\ref{main-penallim}) also holds along the constant clock:

\begin{thm}
Suppose that (\ref{RV-rq}) holds. Then for every $x\in \R$ and $s\ge 0$,
 \begin{align}
 \label{t-penal}
   \lim_{t\to \infty}\frac{\P_x[F_s\cdot \Gamma_t^\mu]}{\P_x[\Gamma_t^\mu]}=\P_x\left[F_s\cdot\frac{\varphi^\mu(X_s)\Gamma_s^\mu}{\varphi^\mu(x)} \right]\qquad \text{for}\ F_s\in b\F_s.
 \end{align}
Here, $\varphi^\mu$ is the function appearing in Theorem \ref{mainthm1}.
\end{thm}
\begin{proof}
Rewriting Proposition \ref{prop3.2} using the density of the exponential distribution, we have
\begin{align}
  \lim_{q\to 0+}qr_q(0)\int_0^\infty e^{-qt}\P_x[\Gamma_t^\mu]dt=\varphi^\mu(x)\qquad \text{for}\ x\in \R.
\end{align}
Since $t\mapsto \P_x[\Gamma_t^\mu]$ is non-increasing, Karamata's monotone density theorem (see, e.g., Theorem 5.14 of \cite{Kyprianou}) yields
\begin{align}
\label{B4}
  \lim_{t\to \infty} r_{\frac{1}{t}}(0)\P_x[\Gamma_t^\mu]= \frac{\varphi^\mu(x)}{\Gamma(1-\rho)},
\end{align}
where $\Gamma(\cdot)$ denotes Euler's gamma function. Fix $s\ge 0$. For $t>s$, we define
\begin{align}
M_t:=r_{\frac{1}{t}}(0)\P_x[\Gamma_t^\mu\mid \F_s].
\end{align}
By (\ref{B4}) and the multiplicative property of $\Gamma^\mu$, we have
\begin{align}
  \lim_{t\to \infty}M_t=\lim_{t\to \infty}\Gamma_s^\mu\cdot \frac{r_{\frac{1}{t}}(0)}{r_{\frac{1}{t-s}}(0)}\cdot r_{\frac{1}{t-s}}(0)\P_{X_s}[\Gamma_{t-s}^\mu]=\frac{1}{\Gamma(1-\rho)} \varphi^\mu(X_s)\Gamma_s^\mu\qquad \P_x\text{-a.s.}
\end{align}
By (\ref{B4}) and Proposition \ref{prop3.4}, we have
\begin{align}
\lim_{t\to \infty}\P_x[M_t]=\lim_{t\to \infty}r_{\frac{1}{t}}(0)\P_x[\Gamma_t^\mu]=\frac{\varphi^\mu(x)}{\Gamma(1-\rho)}=\frac{1}{\Gamma(1-\rho)} \P_x[\varphi^\mu(X_s)\Gamma_s^\mu].
\end{align}
Thus, Scheff\'{e}'s lemma implies
\begin{align}
 r_{\frac{1}{t}}(0)\P_x[\Gamma_t^\mu\mid \F_s]\to \frac{1}{\Gamma(1-\rho)} \varphi^\mu(X_s)\Gamma_s^\mu\qquad \text{in}\ L^1(\P_x).
\end{align}
Therefore, for $F_s\in b\F_s$, we have
\begin{align}
  \lim_{t\to \infty}r_{\frac{1}{t}}(0)\P_x[F_s\cdot \Gamma_t^\mu]=\lim_{t\to \infty}\P_x[F_s\cdot M_t]=\frac{1}{\Gamma(1-\rho)}\P_x[F_s\cdot \varphi^\mu(X_s)\Gamma_s^\mu].
\end{align}
Dividing this equation by (\ref{B4}), we obtain (\ref{t-penal}). This completes the proof.
\end{proof}


\section*{Acknowledgments}
The author is deeply grateful to Prof. Kouji Yano and Prof. V\'{i}ctor M. Rivero for their insightful comments, from which this work benefited greatly. The author acknowledges support from JSPS KAKENHI Grant Number 26KJ1605.

\bibliographystyle{plain}

\end{document}